%% file: graph_groups3.0.tex
\documentclass[reqno]{amsart}

\usepackage{amsmath,amsfonts,amsthm,amssymb,amscd}
\usepackage{hyperref}
\usepackage[alphabetic]{amsrefs}
\usepackage{tikz}
\usepackage{graphicx}
\usepackage{tikz-cd}
\usepackage{enumitem}

\usetikzlibrary{patterns}
\usepackage{apptools}

\newtheorem{theorem}{Theorem}[section]
\newtheorem{lemma}[theorem]{Lemma}
\newtheorem{corollary}[theorem]{Corollary}
\newtheorem{proposition}[theorem]{Proposition}
    
\newcommand*{\textlabel}[2]{%
  \edef\@currentlabel{#1}
  \phantomsection
  #1\label{#2}
}

\theoremstyle{definition}
\newtheorem{definition}[theorem]{Definition}
\newtheorem{convention}[theorem]{Convention}
\newtheorem{example}[theorem]{Example}

\newtheorem*{lemma*}{Lemma}

\theoremstyle{remark}
\newtheorem{remark}[theorem]{Remark}

\numberwithin{equation}{section}

\newcommand{\abs}[1]{\lvert#1\rvert}

\newcommand{\mc}{\mathcal}
\newcommand{\mQ}{\mathcal{Q}}
\newcommand{\czero}{\mathrm{CAT}(0)}
\newcommand{\td}[1]{\tilde{#1}}
\newcommand{\tr}{t}

\begin{document}

\title{Graphical complexes of groups}

\author[T. Prytu{\l}a]{Tomasz Prytu\l a}

\address{Department of applied mathematics and computer science, Technical University of Denmark, Lyngby, Denmark}

\email{tompr@dtu.dk}


\date{\today}


\begin{abstract}
  We introduce graphical complexes of groups, which can be thought of as a generalisation of Coxeter systems with 1-dimensional nerves. We show that these complexes are strictly developable, and we equip the resulting Basic Construction with three structures of non-positive curvature: piecewise linear CAT(0), C(6) graphical small cancellation, and a systolic one. We then use these structures to establish various properties of the fundamental groups of these complexes, such as biautomaticity and Tits Alternative. We isolate an easily checkable condition implying hyperbolicity of the fundamental groups, and we construct some non-hyperbolic examples. We also briefly discuss a parallel theory of C(4)-T(4) graphical complexes of groups and outline their basic properties.
\end{abstract}

\maketitle

\section{Introduction}

Given a simplicial graph $\Gamma$ of girth at least four, consider the associated right-angled Coxeter group $W(\Gamma)$. The Davis complex for $W(\Gamma)$ has a natural structure of a graphical small cancellation complex, in the sense that it is built out of copies of a simplicial cone over $\Gamma$, with two such copies sharing at most a subgraph of $\Gamma$ in each copy. Small cancellation conditions of this complex are controlled by the girth of $\Gamma$. Moreover, this graphical small cancellation structure is compatible with the structure of a simple complex of groups, the latter being more commonly used in the context of Coxeter groups.

This observation leads us to the concept of \emph{graphical complexes of groups}, which can be thought of as a generalisation of Coxeter groups with $1$--dimensional nerves. In this article we introduce and initiate the systematic study of graphical complexes of groups. The highlight of our theory is that one can easily construct examples of graphical complexes of groups, while at the same time they possess a rich geometry which allows one to establish strong properties of their fundamental groups, such as biautomaticity or Tits Alternative. This geometry is based on the interplay between small cancellation techniques, $\czero$ geometry, and systolic geometry.

Formally, a graphical complex of groups is a simple complex of finite groups over a finite, $1$--dimensional poset. Recall that a poset $\mQ$ is $1$--dimensional, if its geometric realisation $\abs{\mQ}$ is a graph. Such a poset has two types of vertices: big and small, and the only possible relation is that a big vertex is larger than a small vertex. Alternatively, a $1$--dimensional poset can be thought of as a bipartite graph. A graphical complex of groups $G(\mQ)$ over a finite, $1$--dimensional poset $\mQ$ consists of a family of finite groups $\{G_v\}_{v \in \mQ}$, called local groups, and a family of injective homomorphisms $\psi_{vw} \colon G_v \to G_w$ for every pair $v \leq w$. The fundamental group $G$ of $G(\mQ)$ is a direct limit of the system $(\{G_v\}_{v \in \mQ}, \{\psi_{vw}\}_{v\leq w})$. We say that $G(\mQ)$ is strictly developable if for every local group the canonical map to the limit is injective. For a strictly developable complex $G(\mQ)$, the so-called Basic Construction is a $2$--dimensional $G$--CW--complex $D(G, C(\abs{\mQ}))$, which can be thought of as an analogue of the Davis complex for Coxeter groups. In particular, like the Davis complex, $D(G, C(\abs{\mQ}))$ is also built out of copies of the simplicial cone $C(\abs{\mQ})$ over $\abs{\mQ}$, and the $G$--action on $D(G, C(\abs{\mQ}))$ is proper and cocompact. For an integer $k \geq4$, we say that $G(\mQ)$ is $k$--huge if the girth of $\abs{\mQ}$ is at least $2k$.

\begin{theorem}[Theorem~\ref{thm:developable}] 
  Let $G(\mQ)$ be a $k$--huge graphical complex of groups, where $k \geq 6$. Then $C(\abs{\mQ})$ admits a piecewise linear metric such that $G(\mQ)$ is non-positively curved. Consequently, $G(\mQ)$ is strictly developable and $D(G, C(\abs{\mQ}))$ admits a $G$--invariant $\czero$ metric.
\end{theorem}

Note that since $G$ acts on $D(G, C(\abs{\mQ}))$ properly and cocompactly, the above theorem implies that $G$ is a $\czero$ group. The following corollary essentially says that, in most cases, $G$ is a $2$--dimensional $\czero$ group.

\begin{corollary}[Corollary~\ref{cor:bredondim} and Convention~\ref{conv:connected}]
  The space $D(G, C(\abs{\mQ}))$ is a cocompact model for the classifying space for proper actions $\underline{E}G$. If $\abs{\mQ}$ contains at least one cycle then $G$ is infinite and the proper geometric dimension of $G$ is equal to $2$. In particular, $G$ is not virtually free.
\end{corollary}

In fact $C(\abs{\mQ})$ (and thus $D(G, C(\abs{\mQ}))$) is metrised with a single shape of cells, which is a triangle with angles $90^{\circ}$, $60^{\circ}$, and $30^{\circ}$, and the shortest edge of length $1$. Such a simple structure of $D(G, C(\abs{\mQ}))$ allows us to prove the Tits Alternative for $G$. Our approach is based on the concept of recurrent $2$--complexes introduced in a recent work of Osajda and Przytycki \cite{OsaPrzy}.

\begin{theorem}[Proposition~\ref{prop:titsalternative}] 
  Let $G(\mQ)$ be a $k$--huge graphical complex of groups for $k \geq 6$ and let $G$ be its fundamental group. Then any subgroup of $G$ is either finite,  virtually $\mathbb{Z}$, virtually  $\mathbb{Z}^2$, or contains a non-abelian free group.
\end{theorem}

Note that the conclusion we obtain is stronger than the `usual' Tits Alternative which concerns only finitely generated subgroups of $G$.

The above results rely almost exclusively on the (piecewise linear) $\czero$ structure of  $G(\mQ )$. We now begin to explore small cancellation and systolic features of graphical complexes of groups. 

\begin{theorem}[Theorem~\ref{thm:gcogissystolic}]
  Let $G(\mQ)$ be a $k$--huge graphical complex of groups for $k \geq 6$ and let $G$ be its fundamental group. Then $D(G, C(\abs{\mQ}))$ with its natural structure of a graphical small cancellation complex satisfies the $C(k)$--condition.
\end{theorem}

For every small cancellation complex $X$ there is an associated dual simplicial complex $W(X)$ called the \emph{Wise complex}. It is shown by Wise in the classical setting \cite{wisesix} and by the author and Osajda in the graphical setting \cite{OsaPry} that if $X$ is simply connected and satisfies the $C(k)$--condition then $W(X)$ is $k$--systolic. The following corollary is straightforward.

\begin{corollary}[Corollary~\ref{cor:wisecomplex}]
  The group $G$ acts geometrically on a $k$--systolic complex $W(D(G, C(\abs{\mQ})))$. Thus $G$ is a $k$--systolic group. Moreover, complexes $D(G, C(\abs{\mQ}))$ and $W(D(G, C(\abs{\mQ})))$ are $G$--homotopy equivalent. 
\end{corollary}

This allows us to use systolic techniques to further study $G$. One immediate consequence of being systolic is that $G$ is biautomatic. Notice that in general, biautomaticity is not known even for $2$--dimensional, piecewise Euclidean $\czero$ groups.

\begin{corollary}[Corollary~\ref{cor:biautomatic}]  
  The group $G$ is biautomatic.
\end{corollary}

The interplay between graphical small cancellation and systolic structures enables us to isolate a condition on a simply connected $C(6)$ graphical small cancellation complex that implies hyperbolicity of any group acting on it geometrically. (It is known that a group acting geometrically on a simply connected $C(7)$ graphical small cancellation complex is hyperbolic \cite{OsaPry}.) Recall that a graphical complex is built out of simplicial cones over finite graphs.
Three such cones form a \emph{proper triple} if their triple intersection is non-empty, and if it is a proper subset of every double intersection.
 
\begin{theorem}[Theorem~\ref{thm:notripleshyperbolicity} and Proposition~\ref{prop:piecewisehyperbolic}]\label{thm:intro.notripleshyperbolicity}
  Let $G(\mQ)$ be a $6$--huge graphical complex of groups and let $G$ be its fundamental group. If $D(G, C(\abs{\mQ}))$ does not contain a proper triple then $G$ is hyperbolic. Moreover, $D(G, C(\abs{\mQ}))$ admits a piecewise hyperbolic $\mathrm{CAT}(-1)$ metric.
\end{theorem}
We remark that proper triples inside $D(G, C(\abs{\mQ}))$ have a simple, easily checkable form (see Lemma~\ref{lem:structureofpropertriple}).

In Section~\ref{sec:examples} we give several examples of graphical complexes of groups. We begin with familiar examples of right-angled Coxeter groups associated to finite graphs, Coxeter groups with $1$--dimensional nerves, and \emph{graphical products of groups}, which are certain graph products of groups that admit a structure of graphical complexes of groups. Using Theorem~\ref{thm:intro.notripleshyperbolicity}, we show that all of these examples have hyperbolic fundamental groups.

We then construct some examples with non-hyperbolic fundamental groups. Let us remark that these examples are genuine graphical complexes of groups, in the sense that both their local groups and the poset structure are more complicated than, e.g., right-angled Coxeter groups, where one considers a poset of simplices of a graph, and where structure maps are inclusions of direct factors into the product. As a byproduct of our construction, we obtain a description of flats in the Basic Construction $D(G, C(\abs{\mQ}))$ for a $6$--huge graphical complex of groups $G(\mQ)$.

In Section~\ref{sec:c4t4} we briefly outline a parallel theory of $C(4)-T(4)$ graphical complexes of groups. It turns out that the condition of not containing proper triples may be seen as a form of a $T(4)$--condition. Combined with the $C(4)$--condition, it allows one to construct non-hyperbolic examples. Most of the theory translates verbatim from the $C(6)$ case, and thus fundamental groups of $C(4)-T(4)$ complexes enjoy the same properties as those of $C(6)$ complexes. The exception is biautomaticity, which follows from the Helly property of $C(4)-T(4)$ small cancellation complexes \cite{Helly}.

We conclude the article with Section~\ref{sec:commques} where we pose some questions regarding the more advanced group theoretic and geometric properties of fundamental groups of graphical complexes of groups.

\subsection*{Acknowledgments} I would like to thank Damian Osajda and Jacek {\'S}wi{\c{a}}tkowski for helpful discussions. I thank Aleksander Pedersen Prytu{\l}a for his assistance during this work. I also thank the Max Planck Institute for Mathematics where part of the work was completed. I was supported by the European Union’s Horizon 2020 research and innovation programme under the Marie Sk{\l}odowska-Curie grant agreement No 713683.

\section{Preliminaries}

\subsection{Systolic simplicial complexes}

Let $X$ be a finite dimensional, uniformly locally finite simplicial complex. We equip $X$ with a CW--topology. Let $X^{(n)}$ denote the $n$--skeleton of $X$. For vertices $v_1, \ldots, v_n \in X^{(0)}$ let $[v_1, \ldots v_n]$ denote the simplex of $X$ spanned by these vertices. We say that $X$ is \emph{flag} if every set of vertices in $X^{(0)}$ pairwise connected by edges spans a simplex of $X$. A map of simplicial complexes $f \colon X \to Y$ is \emph{simplicial} if $f(X^{(0)}) \subset Y^{(0)}$ and whenever vertices $v_1, \ldots v_n$ span a simplex of $X$, then so do $f(v_1), \ldots, f(v_n)$ and $f$ maps $[v_1, \ldots v_n]$ linearly onto $[f(v_1), \ldots f(v_n)]$. A simplicial map is \emph{non-degenerate} if it is injective on each open simplex. An \emph{immersion} is a simplicial map that is locally injective.

A \emph{path} (resp.\ \emph{cycle}) in $X$ is a non-degenerate simplicial map from a subdivided interval (resp.\ circle) into $X$. A path or cycle is \emph{embedded} if the respective map from an interval or circle is injective. In this case we can identify a path or a cycle with its image in $X$. The \emph{length} of a path or a cycle is the number of edges in its domain. A \emph{diagonal} of an embedded cycle $C$ is an edge connecting two non-consecutive vertices of $C$. We equip $X^{(0)}$ with a metric such that $d(v,w)$ is the length of the shortest path joining $v$ and $w$. We call this metric the \emph{edge-path} metric.

A \emph{simplicial join} of simplicial complexes $X$ and $Y$ is a simplicial complex $X \ast Y$ with vertex set $X^{(0)} \sqcup Y^{(0)}$ and simplices $[v_0,\ldots, v_n, w_0, \ldots, w_m]$ where $[v_0,\ldots, v_n]$ is a simplex of $X$ and $[w_0,\ldots, w_m]$ is a simplex of $Y$.

The $\emph{link}$ of a vertex $v \in X^{(0)}$ is a subcomplex $X_v \subset X$ which consists of all simplices that are disjoint from $v$ but that together with $v$ span a simplex of $X$. The \emph{star} of $v$ is a subcomplex of $X$ that consists of all simplices that contain $v$.  Note that the star of $v$ in $X$ is naturally isomorphic with the join $X_v \ast \{v\}$. A \emph{valence} of a vertex $v$ is the number of edges in $X$ that are incident to $v$.

\begin{definition}\label{def:klarge}
  Let $k \geq 4$ be an integer. A simplicial complex $X$ is \emph{$k$--large} if it is flag and if every embedded cycle of length less than $k$ in $X$ has a diagonal.
\end{definition}

\begin{remark}
  Observe that if $k \geq l$ then `$k$--large' implies `$l$--large'.
\end{remark}

\begin{definition}
  A simplicial complex $X$ is \emph{$k$--systolic} if it is simply connected and the link of every vertex of $X$ is $k$--large. We abbreviate $6$--systolic to \emph{systolic}.

  A group $G$ is \emph{$k$--systolic} if it acts geometrically (i.e.,\ properly, cocompactly, and by simplicial automorphisms) on a $k$--systolic complex.
\end{definition}

We refer the reader to \cite{JS2} for a background on systolic complexes and groups.

\begin{definition}
  A \emph{graph} is a $1$--dimensional simplicial complex. (Such graphs are often called `simplicial graphs'.) The \emph{girth} of a graph is the length of its shortest cycle. If a graph does not contain any cycles, we declare its girth to be $\infty$.
\end{definition}

\begin{remark}
  A graph is $k$--large if and only if its girth is at least $k$.
\end{remark}

\subsection{Graphical small cancellation theory}

Intuitively speaking, graphical small cancellation theory studies complexes which are built out of cones over finite graphs, glued together along subgraphs of these graphs.
The small cancellation condition applies only to paths which are contained in two different such cones. Inside a single cone, there can be arbitrarily large cancellations.

Recall that a \emph{cone} over a graph $\Gamma$ is a space $C(\Gamma)$ given by \[ C(\Gamma) =\Gamma \times [0,1]/ (x,1) \sim (y,1).\] Notice that $C(\Gamma)$ has a natural structure of a simplicial complex which is a simplicial join of $\Gamma$ and the cone vertex. Moreover, $\Gamma$ is a subcomplex of $ C(\Gamma)$.

\begin{definition}[Graphical complex]\label{def:grcomplex}
  Let $\Theta$ be a graph and let $\Gamma=\bigsqcup_{i\in I} \Gamma_i$ be a disjoint union of countably many finite graphs. Consider an immersion $\phi= \sqcup \phi_i \colon \Gamma \to \Theta$. A \emph{graphical complex} $X$ is obtained by for every $i \in I$ gluing a cone $C(\Gamma_i)$ to $\Theta$ along the map $\phi_i$, i.e.,\ 
  \begin{equation*}
    X = \Theta \cup_{\phi} \bigsqcup_{i \in I} C(\Gamma_i).
  \end{equation*} 
\end{definition}

An image of $C(\Gamma_i)$ in $X$ is called a \emph{cone-cell}. Notice that $X$ has a natural structure of a simplicial complex. However, intuitively one can think of cone-cells as `$2$--cells' of $X$ and of $\Theta$ as the `$1$--skeleton' of $X$.

\begin{definition} 
  A group $G$ acts on a graphical complex $X$ if it acts simplicially on $\Theta$ and maps cone-cells isomorphically onto cone-cells.
\end{definition}

\begin{definition}[Piece]
  An immersed path $P \to \Theta$ is a \emph{piece} if there are two factorisations of $P$ as follows
  \begin{equation*}
    \begin{tikzcd}
              P \arrow{r} \arrow{d} & \Gamma_i \arrow{d} \\
      \Gamma_j \arrow{ru} \arrow{r} &  \Theta, 
    \end{tikzcd}
  \end{equation*}
  and there does not exist an isomorphism $\Gamma_j \to \Gamma_i$ making the above diagram commute.
\end{definition}

\begin{definition}[$C(k)$--condition]
  Let $k \geq 2$ be an integer. A graphical complex $X$ satisfies the \emph{$C(k)$--condition} if no embedded cycle $C \to \Theta$ which factors through some $\Gamma_i \to \Theta$ is a concatenation of less than $k$ pieces.
\end{definition}

\begin{definition}[Wise complex]\label{def:wisecomplex}
  Let $X$ be a graphical complex such that $\phi \colon \Gamma \to \Theta$ is surjective. Define simplicial complex $W(X)$ as the nerve of the covering of $X$ by its cone-cells, i.e., we have 
  \begin{equation*} 
    \begin{split}
      W(X)^{(0)} & = \{C(\Gamma_i)\}_{i \in I}, \\
      W(X)^{(n)} & = \{ (C(\Gamma_0),\ldots, C(\Gamma_n)) \mid \Gamma_0 \cap \ldots \cap \Gamma_n \neq \emptyset\}.
    \end{split}
  \end{equation*}
\end{definition}

\begin{theorem}\cite[Theorem~F, Theorem~7.9, and Theorem~7.10]{OsaPry}\label{thm:graphicalissystolic}
  Let $X$ be simply connected graphical complex and let $k \geq 6$ be an integer. If $X$ satisfies the $C(k)$--condition then $W(X)$ is $k$--systolic.

  If $G$ acts on $X$ then $G$ acts on $W(X)$, and $X$ and $W(X)$ are $G$--homotopy equivalent. One action is proper and/or cocompact if and only if the other is so.
\end{theorem}

\section{Graphical complexes of groups}

Let $\mQ$ be a poset. We say that $\mQ$ is $1$--dimensional if its geometric realisation $\abs{\mQ}$ is a graph. We say that $\mQ$ is connected if $\abs{\mQ}$ is connected. From now on we assume that $\mQ$ is finite, $1$--dimensional, and connected. We will abuse terminology and call elements of $\mQ$ \emph{vertices}. Note that $\mc Q$ has two types of vertices: \emph{big} and \emph{small}, defined in the obvious way. An example of such a poset $\mQ$ is shown in Figure~\ref{fig:1dimposet}.

Let $C(\mQ)$ denote a poset obtained by adding to $\mQ$ an additional element $c$ and declaring $c \leq v$ for every $v \in \mQ$. We call $c$ the \emph{cone vertex} of $C(\mQ)$. Note that we have $\abs{C(\mQ)} \cong C(\abs{\mQ})$ and the inclusion of posets $ \mQ \subset C(\mQ)$ induces the canonical inclusion of the graph $\abs{Q}$ into the cone $C(\abs{\mQ})$. 

We introduce the main object of our study. 

\begin{definition}[Graphical complex of groups]\label{def:gcog}
  Suppose that $\mQ$ is a finite $1$--dimensional poset. A \emph{graphical complex of groups} $G(\mQ)$ consists of 
  \begin{itemize}
    \item a collection of non-trivial finite groups $\{G_v\}_{v \in \mQ}$, called the \emph{local groups}, 
    \item a collection of maps $\psi_{vw} \colon G_v \to G_w$ for every pair $v \leq w$,
  \end{itemize} satisfying the following conditions:
  \begin{enumerate}
    \item every $\psi_{vw} \colon G_v \to G_w$ is a proper inclusion,
    \item if $v_1 \leq w$ and $v_2 \leq w$ then $\psi_{v_1w}(G_{v_1}) \cap \psi_{v_2w}(G_{v_2})= \{e\}.$
  \end{enumerate}

  \noindent \textlabel{$(\ast)$}{ast:coneconvention} Additionally, as a part of the definition, one considers the poset $C(\mQ)$ with $\mQ$ being its subposet. To the cone vertex $c \in C(\mQ)$ one associates the trivial group as the local group, and to every pair $c \leq v$ with $v \in \mQ$ one associates the inclusion of the trivial group into $G_v$.
\end{definition}

\begin{remark} 
  The above definition is a special case of a \emph{simple complex of groups} as defined in \cite{BH}. In our case many standard constructions for simple complexes of groups get substantially simplified, and therefore we present them in a way adjusted to our setting (in particular, we use a different order convention than in \cite{BH}, where map $\psi_{vw}$ maps $G_w$ to $G_v$). 

  We remark further that part~\hyperref[ast:coneconvention]{$(\ast)$} of Definition~\ref{def:gcog} does not carry any essential information, and is added in order to make the material in this section consistent with the literature on simple complexes of groups (modulo the inconsistent order convention described above). 
\end{remark}

\begin{definition}
  Define the \emph{fundamental group} of $G(\mQ)$ as \[G= \underset{\longrightarrow}{\mathrm{lim}}\{G_v, \psi_{vw} \}.\]In other words, $G$ is isomorphic to the free product of groups $G_v$ for $v \in \mQ$, divided by the relations $\psi_{vw}(G_v)=G_w$ for each pair $v \leq w$ in $\mQ$.
\end{definition}

\begin{definition} 
  We say that $G(\mQ)$ is \emph{strictly developable} if for every $v \in \mQ$ the canonical map $G_v \to G= \underset{\longrightarrow}{\mathrm{lim}}\{G_v, \psi_{vw} \}$ is injective. Note that in this case one can identify group $G_v$ with its image in $G$.
\end{definition}

Observe that if $\mQ$ contains a vertex $v$ for which there is a unique $w$ with $v \leq w$, then removing from $G(\mQ)$ the vertex $v$ and the local group $G_v$ does not affect whether or not $G(\mQ)$ is developable, nor does it change its fundamental group. Moreover, if $\mQ$ has a disconnecting vertex (i.e.,\ a vertex whose removal will result in a non-connected poset), then the fundamental group $G(\mQ)$ may be written as a non-trivial amalgamated product over a finite group. 

\begin{convention}\label{conv:connected}
  In the light of the above discussion, with a little loss of generality, we will assume from now on that $\mQ$ has no vertices of valence one and no disconnecting vertices.

  Furthermore, we will assume that $\mQ$ contains more than one vertex. Note that, together with the first assumption, this implies that $\abs{\mQ}$ contains at least one cycle.
\end{convention}

\input{poset.tex}

\begin{definition}
  For a vertex $v \in \mQ$ define $\mQ_{\geq v}$ (resp.\ $\mQ_{> v}$) to be the subposet of $\mQ$ consisting of all vertices greater than or equal to (resp.\ strictly greater than) $v$. Subposets $\mQ_{\leq v}$ and $\mQ_{< v}$ are defined analogously.
\end{definition}

Observe that if $v \in \mQ$ is a small vertex then $\abs{\mQ_{\geq v}}$ consists of all edges of $\abs{\mQ}$ which are incident to $v$. If $w \in \mQ$ is a big vertex then $\abs{\mQ_{\geq w}} =\{w\}.$ Note that if $v \leq w$ then clearly $w \in \abs{\mQ_{\geq v}}$ (see Figure~\ref{fig:1dimposet}).  

\begin{definition}[Basic Construction]\label{def:basicconstruction}
  Let $G(\mQ)$ be a graphical complex of groups. Define the \emph{Basic Construction} as \[D(G, C(\abs{\mQ}))= G \times C(\abs{\mQ}) / \sim\] where $(g_1, x_1) \sim (g_2, x_2)$ if and only if $x_1=x_2$ and
  \begin{itemize}
    \item $g_1^{-1}g_2 \in G_w$ if $x_1=w$ for some big vertex $w \in \abs{\mQ}$,
    \item $g_1^{-1}g_2 \in G_v$ if $x_1 \in \abs{\mQ_{\geq v}}$ for some small vertex $v \in \abs{\mQ}$.
  \end{itemize}
  Let $[g, x]$ denote the equivalence class of $(g,x)$.
\end{definition}

Observe that $D(G, C(\abs{\mQ}))$ has a natural structure of a simplicial complex. The group $G$ acts on $D(G, C(\abs{\mQ}))$ by $g_1 \cdot [g,x] = [g_1g,x]$. The stabilisers of this action are $G$--conjugates of local groups and the quotient is isomorphic to $[e,C(\abs{\mQ})] \cong C(\abs{\mQ})$. Since the quotient may be seen as a subcomplex of $D(G, C(\abs{\mQ}))$ we call it a \emph{strict fundamental domain}.

\subsection{Graphical structure of the Basic Construction}\label{subsec:graphicalstructure}
Note that $D(G, C(\abs{\mQ}))$ has a natural structure of a graphical complex. The `$1$--skeleton' is given by 
\[D(G, \abs{\mQ})= \{[g,x] \in D(G, C(\abs{\mQ})) \mid x \in \abs{\mQ}\}\]  (note that the equivalence relation $\sim$ does not identify any points which are not contained in $\abs{\mQ})$. The cone-cells are given by $\{g\} \times C(\abs{\mQ})$ attached along the map $\{g\} \times \abs{\mQ} \to  [g,\abs{\mQ}] \subset D(G, \abs{\mQ}).$ Clearly the $G$--action on $D(G,  C(\abs{\mQ}))$ preserves the graphical structure.

\subsection{Three types of vertices of the Basic Construction}\label{subsec:threetypes}

The simplicial complex $D(G,  C(\abs{\mQ}))$ has three types of vertices: small, big, and \emph{cone} vertices (i.e.,\ cone vertices of cones $[g,C(\abs{\mQ})]$ for $g \in G$). This is because the equivalence relation $\sim$ on $G \times C(\abs{\mQ})$ preserves the type of vertices of $C(\abs{\mQ})$. Moreover, every $2$--simplex of $D(G,  C(\abs{\mQ}))$ has exactly one vertex of each type. It is easy to see that the $G$--action on $D(G, C(\abs{\mQ}))$ preserves the type of vertices.

\subsection{k--huge graphical complexes of groups}

In this section we define $k$--huge posets and $k$--huge graphical complexes of groups, whose study occupies the remainder of the article.

\begin{definition}[$k$--hugeness]
Let $\mQ$ be a finite $1$--dimensional poset and let $k \geq 2$ be an integer. We say that $\mQ$ is \emph{$k$--huge} if every embedded cycle in $\abs{\mQ}$ contains at least $k$ small (equivalently, big) vertices. Note that this happens if and only if $\abs{\mQ}$ is $2k$--large in the sense of Definition~\ref{def:klarge} (i.e., the girth of $\abs{\mQ}$ is at least $2k$). We say that $G(\mQ)$ is $k$--huge if $\mQ$ is $k$--huge.
\end{definition}

\begin{remark}
  There is a notion of a \emph{$k$--large complex of groups}, originally defined in \cite{JS2}. We remark that this notion is not equivalent to our definition of $k$--hugeness ($=2k$--largeness). However, these two definitions are similar, as they both give conditions ensuring strict developability and non-positive curvature-like features of the Basic Construction.
\end{remark}

\section{Developability and Tits Alternative}

\subsection{Strict developability}

\begin{theorem}\label{thm:developable} 
  Let $G(\mQ)$ be a $k$--huge graphical complex of groups, where $k \geq 6$. Then $C(\abs{\mQ})$ admits a piecewise linear metric such that $G(\mQ)$ is non-positively curved. Consequently, $G(\mQ)$ is strictly developable and $D(G, C(\abs{\mQ}))$ admits a $G$--invariant $\czero$ metric.
\end{theorem}

Since the $G$--action on $D(G, C(\abs{\mQ}))$ is by construction proper and cocompact, we obtain that $G$ is a $\czero$ group. For a definition and some properties of $\czero$ spaces and groups we refer the reader to \cite{BH}.

Before proving Theorem~\ref{thm:developable} we present the following corollary. Let $\underline{E}G$ denote the classifying space for proper actions for $G$ and let $\underline{\mathrm{gd}}G$ denote the proper geometric dimension of $G$, i.e.,\ the minimal dimension of a model for $\underline{E}G$ (see \cite{PePry} for some background on these notions).

\begin{corollary}\label{cor:bredondim}
  The space $D(G, C(\abs{\mQ}))$ is a cocompact model for $\underline{E}G$ and we have $\underline{\mathrm{gd}}G = 2$. In particular, $G$ is not virtually free.  
\end{corollary}

\begin{proof}
  The first assertion follows from the Fixed Point Theorem for $\czero$ spaces \cite[Corollary~II.2.8]{BH}. To show the second assertion, first observe that since $D(G, C(\abs{\mQ}))$ is a $2$--dimensional model for $\underline{E}G$, we have $\underline{\mathrm{gd}}G \leq 2$. On the other hand, since $C(\mQ)_{>c}=\mQ$ and $\widetilde{H}^1(\abs{\mQ}, \mathbb{Z}) \neq 0$ (see Convention~\ref{conv:connected}) we obtain by \cite[Proposition~3.6]{PePry} that $\underline{\mathrm{gd}}G \geq 2$.
\end{proof}

In order to prove Theorem~\ref{thm:developable}, we metrise complex $C(\abs{\mQ})$.

\begin{definition}\label{def:metriconthecone}
  To every triangle ($2$--simplex) in $C(\abs{\mQ})$ we assign the unique Euclidean metric such that:
  \begin{enumerate}
    \item the angle at the small vertex is $\frac{\pi}{2}$,
    \item the angle at the big vertex is $\frac{\pi}{3}$,
     \item the angle at the cone vertex is $\frac{\pi}{6}$,  
     \item the length of the shortest edge is equal to $1$.
  \end{enumerate}
  By Subsection~\ref{subsec:threetypes}, this gives a well-defined metric on $C(\abs{\mQ})$. Note that, when restricted to $\abs{\mQ}$, this metric agrees with the edge-path metric. An example of $C(\abs{\mQ})$ equipped with this metric is presented in Figure~\ref{fig:metric}.
\end{definition}

It remains to show that $G(\mQ)$ with $C(\abs{\mQ})$ metrised as above is non-positively curved in the sense of \cite{BH}. We need the following definition.

\begin{definition}\label{def:angularmetric}
  Let $X$ be a $2$--dimensional simplicial complex equipped with a piecewise Euclidean metric. Then for any vertex $v \in X$ its link $X_v$ carries the angular metric, where the length of an edge $[v_1,v_2]$ is the angle at $v$ in the triangle $[v, v_1, v_2]$.
\end{definition}

\begin{definition}\label{def:linkcondition}
  A $2$--dimensional simplicial complex $X$ equipped with a piecewise Euclidean metric is locally $\mathrm{CAT}(0)$ if for every vertex $v  \in X$ the link $X_v$ has girth at least $2\pi$ with respect to the angular metric.
\end{definition}

\input{metric.tex}

\begin{proof}[Proof of Theorem~\ref{thm:developable}]
  In order to show that $G(\mQ)$ is non-positively curved in the sense of \cite{BH} we need to show that local developments at vertices of $C(\abs{\mQ})$ are locally $\czero$. We refer the reader to \cite[II.12.24]{BH} for a definition of local development. For a vertex $v \in C(\abs{\mQ})$ let $\mathrm{St}(\tilde{v})$ denote its closed star in the local development, and let $\mathrm{Lk}(\tilde{v})$ denote its link in $\mathrm{St}(\tilde{v})$. Recall that we have $\mathrm{St}(\tilde{v}) \cong \mathrm{Lk}(\tilde{v}) \ast \tilde{v}$. Examples of local developments are shown in Figure~\ref{fig:metric}.

  \begin{enumerate}
    \item{Local development at cone vertex $c \in C(\abs{\mQ})$.} We have \[\mathrm{St}(\tilde{c})\cong C(\abs{\mQ})\] and \[\mathrm{Lk}(\tilde{c}) \cong \abs{\mQ}.\]  
    Every edge in $\abs{\mQ}$ has angular length $\frac{\pi}{6}$. Since $\mQ$ is $k$--huge with $k \geq 6 $ we get that any embedded cycle consists of at least $2k$ edges and thus the girth of  $\abs{\mQ}$ with respect to the angular metric is at least \[2k \cdot \frac{\pi}{6} \geq 12 \cdot \frac{\pi}{6} \geq 2\pi.\] 

    \item{Local development at small vertex $v \in C(\abs{\mQ})$.} We have \[\mathrm{St}(\tilde{v}) \cong \abs{Q_{\geq v}}  \ast \{g\cdot c \mid  g \in G_v\}\] 
    and \[\mathrm{Lk}(\tilde{v}) \cong \abs{Q_{> v}}  \ast \{g\cdot c \mid  g \in G_v\}.\] The embedded cycles  are given by $(g_1 \cdot c, w_1, g_2 \cdot c, w_2)$ with $v \leq w_1$ and $v \leq w_2$. Since  all edges have  angular length $\frac{\pi}{2}$ we get that the girth of $\mathrm{Lk}(\tilde{v})$ is equal to \[4 \cdot \frac{\pi}{2}= 2\pi.\]

    \item{Local development at big vertex $w \in C(\abs{\mQ})$.} We have \[\mathrm{St}(\tilde{w}) \cong \{w\}  \ast D(G_w, \abs{Q_{<w}} )\] where $D(G_w, \abs{Q_{<w}})  = G_w \times C(\{v_1, \ldots, v_k\}) /\sim$ such that $v_1, \ldots, v_k$ are small vertices in $\abs{Q_{<w}}$ and $(g_1, v_1) \sim  (g_2, v_2)$ if and only if $v_1=v_2$ and $g_1^{-1}g_2 \in G_{v_1} \leq G_w$. Consequently, we  have \[\mathrm{Lk}(\tilde{w}) \cong D(G_w, \abs{Q_{<w}}).\] One easily verifies that every embedded cycle in $\mathrm{Lk}(\tilde{w})$ has even length. Since every edge has angular length $\frac{\pi}{3}$ we have to exclude the existence of a $4$--cycle (there are no $2$--cycles because by definition $\mathrm{Lk}(\tilde{w})$ is simplicial). A putative $4$--cycle has a form \[([g_1, c],\ [g_1, v_1]=[g_2,v_1],\ [g_2, c],\ [g_2, v_2]=[g_1,v_2]).\] We have $(g_1, v_1) \sim (g_2, v_1)$ so $g_1^{-1}g_2 \in G_{v_1}=\psi_{v_1w}(G_{v_1})$ and $(g_1, v_2) \sim (g_2, v_2)$ so $g_1^{-1}g_2 \in G_{v_2}=\psi_{v_2w}(G_{v_2})$ which contradicts the assumption of Definition~\ref{def:gcog} that  $\psi_{v_1w}(G_{v_1})\cap \psi_{v_2w}(G_{v_2}) = \{e\}$. 
  \end{enumerate}

  This shows that $G(\mQ)$ is non-positively curved and thus it is strictly developable. The metric on $C(\abs{\mQ})$ described in Definition~\ref{def:metriconthecone} induces a metric on $D(G, C(\abs{\mQ}))$ given by declaring every cone $[g, C(\abs{\mQ})]$ isometric to $[e,C(\abs{\mQ}))] \cong  C(\abs{\mQ})$ via the multiplication by $g$. Note that by construction this metric is $G$--invariant. Since $G(\mQ)$ is non-positively curved, this metric is $\czero$.
\end{proof}

\begin{remark}\label{rem:linksandstars} 
  By the above construction, for any vertex $[g,v] \in D(G, C(\abs{\mQ}))$, the star (resp.\ link) of $[g,v]$ in $D(G, C(\abs{\mQ}))$ is isomorphic (and isometric) to the star $\mathrm{St}(\tilde{v})$ (resp.\ link $\mathrm{Lk}(\tilde{v})$) in the local development.
\end{remark}

\subsection{Geodesic completeness}

For a metric space $X$, a \emph{geodesic} is an isometric embedding $[a,b] \to X$, where $[a,b] \subseteq \mathbb{R}$ is an interval (we allow $a,b = \infty$).
We say that $X$ is \emph{geodesically complete} if every geodesic $[a,b] \to X$ can be extended to a geodesic line $\mathbb{R} \to X$.

\begin{proposition}
  Let $G(\mQ)$ be a $k$--huge graphical complex of groups, where $k \geq 6$ and let $G$ be the fundamental group of $G(\mQ)$. Then $D(G, C(\abs{\mQ}))$ with its $\czero$ metric is geodesically complete. 
\end{proposition}

\begin{proof}
  Since $D(G, C(\abs{\mQ}))$ is a piecewise Euclidean simplicial $2$--complex, being geodesically complete is equivalent to having no \emph{free faces} i.e., simplices which are contained in exactly one higher dimensional simplex \cite[Proposition~5.10]{BH}. We will show that in vertex links, every vertex has valence at least $2$. One easily verifies that this implies that there are no free faces in $D(G, C(\abs{\mQ}))$. By the construction of metric on $D(G, C(\abs{\mQ}))$, vertex links are isomorphic to the local developments described in the proof of Theorem~\ref{thm:developable}  (see Remark~\ref{rem:linksandstars}).
  \begin{enumerate}

    \item For a cone vertex we have $\mathrm{Lk}(\tilde{c}) \cong \abs{\mQ}$. By Convention~\ref{conv:connected} we have that $\abs{\mQ}$ does not contain any vertices of valence $0$ or $1$.

    \item For a small vertex we have $\mathrm{Lk}(\tilde{v}) \cong \abs{Q_{> v}}  \ast \{g\cdot c \mid  g \in G_v\}$. Since both $\abs{Q_{> v}}$ and  $\{g\cdot c \mid  g \in G_v\}$ contain at least two elements each (by Convention~\ref{conv:connected} and Definition~\ref{def:gcog} respectively), every vertex of $\abs{Q_{> v}}  \ast \{g\cdot c \mid  g \in G_v\}$ has valence at least $2$.
    
    \item For a big vertex we have $\mathrm{Lk}(\tilde{w}) \cong D(G_w, \abs{Q_{<w}}) = G_w \times C(\{v_1, \ldots, v_k\})/ \sim$. By Convention~\ref{conv:connected} there are at least two vertices in $ C(\{v_1, \ldots, v_k\})$ and thus cone vertices of $G_w \times C(\{v_1, \ldots, v_k\})/ \sim$ have valence at least $2$. 

    Now consider a vertex $[g_i,v_j] \in G_w \times C(\{v_1, \ldots, v_k\})/ \sim.$ Coset $g_iG_j$ contains at least two elements, and so let $g_i'\in g_iG_j$ be the element different from $g_i$. We clearly have $[g_i,v_j]=[g_i',v_j]$ and thus vertex $[g_i,v_j]$ is incident to edges $[ [g_i, c], [g_i,v_j]]$ and  $[ [g_i', c], [g_i',v_j]]$.\qedhere

  \end{enumerate}
\end{proof}

\begin{corollary}
  Let $G(\mQ)$ be a $k$--huge graphical complex of groups, where $k \geq 6$ and let $G$ be the fundamental group of $G(\mQ)$. Then $D(G, C(\abs{\mQ}))$ has infinite diameter, and thus $G$ is infinite. 
\end{corollary}

We remark that the above corollary follows also from Corollary~\ref{cor:bredondim}. However, the argument using geodesic completeness is somehow easier to picture. Moreover, this approach allows us to describe a certain family of elements of $G$ of infinite order, by explicitly constructing their axes in $D(G, C(\abs{\mQ}))$.

\begin{definition}
  Two small vertices $v_1,v_2 \in \mQ$ are called \emph{antipodal} if the edge-path distance $d(v_1,v_2)$ in $\abs{\mQ}$ is at least $6$. (This happens if and only if every path between $v_1$ and $v_2$ in $\abs{\mQ}$ passes through at least $3$ big vertices.)
\end{definition}

\begin{proposition}\label{prop:inforderelements}
  Let $G(\mQ)$ be a $k$--huge graphical complex of groups, where $k \geq 6$ and let $G$ be the fundamental group of $G(\mQ)$. Let $v_1, v_2 \in \abs{\mQ}$ be a pair of antipodal vertices and let $g_1 \in G_{v_1}$ and $g_2 \in G_{v_2}$ be non-trivial elements. Then $g_2g_1$ is of infinite order.
\end{proposition}

\begin{proof}Let $\gamma$ denote a path in $C(\abs{\mQ})$ formed by edges $[v_1,c]$ and $[c,v_2]$. Let $\tilde{\gamma}$ denote the path $\gamma$ with the reversed orientation. Consider a path $\alpha$ in $D(G, C(\abs{\mQ}))$ given by the concatenation of the following paths
\[ \ldots \ast [g_1^{-1}g_2^{-1}g_1^{-1}g_2^{-1}, \gamma] \ast [g_1^{-1}g_2^{-1}g_1^{-1}, \tilde{\gamma}] \ast [g_1^{-1}g_2^{-1}, \gamma] \ast [g_1^{-1}, \tilde{\gamma}]  \ast \]
\[ \ast [e,\gamma] \ast [g_2, \tilde{\gamma}] \ast [g_2g_1,\gamma] \ast [g_2g_1g_2, \tilde{\gamma}] \ast [g_2g_1g_2g_1, \gamma] \ast \ldots \] 
We will show that $\alpha$ is a geodesic path. This will imply that all elements $(g_2g_1)^{n}$ are different from one another, and thus that $g_2g_1$ is of infinite order. By construction, $g_2g_1$ acts on $\alpha$ as a translation.

To show that $\alpha$ is geodesic it is enough to show that it is a local geodesic. First we show that $\gamma$ is a local geodesic. To see that, observe that since $v_1$ and $v_2$ are antipodal, the   edge-path distance between them in the link of $c$ is at least $6$. Since every edge has length $\frac{\pi}{6}$ with respect to the angular metric, the angular distance is at least $\pi$. This shows that $\gamma$ is a local geodesic. Consequently, so is  $\tilde{\gamma}$ and  any of their $G$--translates. 

It remains to show that $\alpha$ is a local geodesic at the `concatenation points'. At these points $\alpha$ consists of edges $ [[h_1,c], [h_1,v_i]]$  and $[[h_2,v_i], [h_2,c]]$ meeting at vertex $[h_1,v_i]=[h_2,v_i]$ (for a fixed $i \in \{1,2\}$). The edge-path distance between $[h_1,c]$  and $[h_2,c]$ in the link of $[h_1,v_i]=[h_2,v_i]$ is equal to $2$. Since every edge has angular length $\frac{\pi}{2}$, the angular distance between $[h_1,c]$  and $[h_2,c]$ is equal to $\pi$.   
\end{proof}

\subsection{Tits Alternative}

\begin{proposition}[Tits Alternative]\label{prop:titsalternative} 
  Let $G(\mQ)$ be a $k$--huge graphical complex of groups for $k \geq 6$ and let $G$ be its fundamental group. Then any subgroup of $G$ is either finite,  virtually $\mathbb{Z}$, virtually $\mathbb{Z}^2$, or contains a non-abelian free group.
\end{proposition}

We would like to point out that Proposition~\ref{prop:titsalternative} gives a partial answer to a question of Osajda and Przytycki regarding the Tits Alternative for small cancellation groups \cite[page 4]{OsaPrzy}. Proposition~\ref{prop:titsalternative} will follow from \cite[Main Theorem]{OsaPrzy} which requires the $\czero$ metric on $D(G, C(\abs{\mQ}))$ to be \emph{recurrent}. We refer the reader to \cite{OsaPrzy} for a definition of a recurrent complex. 

\begin{lemma}\label{lem:recurrent}
  The complex $D(G, C(\abs{\mQ}))$ is recurrent with respect to any automorphism group.
\end{lemma}

\begin{proof}
  By \cite[Remark~2.3 and Example~2.5]{OsaPrzy} it is enough to construct a simplicial map from $D(G, C(\abs{\mQ}))$ to a simplicial complex $T$ that consists of a single triangle with angles $\frac{\pi}{2}$, $\frac{\pi}{3}$, and $\frac{\pi}{6}$, and the shortest edge of length $1$, which is an isometry on each triangle of $D(G, C(\abs{\mQ}))$.

  Call the vertices of $T$ small, big, and cone, such that the angles at these vertices are respectively $\frac{\pi}{2}$, $\frac{\pi}{3}$, and $\frac{\pi}{6}$. Define the simplicial map $D(G, C(\abs{\mQ})) \to T$ by sending all vertices of a given type to the unique vertex of $T$ of that type. By Subsection~\ref{subsec:threetypes} this assignment gives a well-defined map which is an isometry on each simplex of $D(G, C(\abs{\mQ}))$.
\end{proof}

\begin{proof}[Proof of Proposition~\ref{prop:titsalternative}] 
  The proof is a straightforward application of \cite[Main Theorem]{OsaPrzy}. We have to verify its assumptions. Clearly $G$ acts on $D(G, C(\abs{\mQ}))$ properly by isometries and without inversions. Since the action is cocompact, every finite subgroup of $G$ is subconjugate to some $G_v$ for $v \in \mQ$, and thus there is a uniform bound on the order of finite subgroups. Finally, by Lemma~\ref{lem:recurrent} we get that $D(G, C(\abs{\mQ}))$ is recurrent with respect to $G$. This implies the claim for all finitely generated subgroups of $G$. Then \cite[Lemma~5.1]{OsaPrzy} implies the claim for all subgroups of $G$.      
\end{proof}

\section{Graphical complexes of groups are systolic}

\begin{theorem}\label{thm:gcogissystolic}
  Let $G(\mQ)$ be a $k$--huge graphical complex of groups for $k \geq 6$ and let $G$ be its fundamental group. Then $D(G, C(\abs{\mQ}))$ with the structure of a graphical complex described in Subsection~\ref{subsec:graphicalstructure} satisfies the $C(k)$--condition.
\end{theorem}

\begin{proof}
  Observe that a path $P \to D(G, \abs{\mQ})$ is a piece if and only if its image is contained in $\abs{\mQ_{\geq v}}$ for some small vertex $v \in \mQ$ (seen as a subset of $[g, \abs{\mQ}]$ for some $g \in G$). Since any $\abs{\mQ_{\geq v}}$ has diameter $2$ (with respect to edge-path metric), any piece has length at most $2$. On the other hand, by the $k$--hugeness assumption, the girth of $\abs{\mQ}$ is at least $2k$, and thus any embedded cycle which factors through $[g, \abs{\mQ}]$ for some $g \in G$ is a concatenation of at least $k$ pieces.
\end{proof}

\begin{remark}
  The above proof shows that for any piece $P \to D(G, C(\abs{\mQ}))$ we have $\abs{P} \leq \frac{1}{k} \cdot \mathrm{girth}(\abs{Q})$. Thus $D(G, C(\abs{\mQ}))$ satisfies two slightly stronger small cancellation conditions: $C'(\frac{1}{k-1})$ and $B(k)$ (see \cites{wisecubu, wisehier}; the latter condition is defined only for $k$ even).
\end{remark}

\begin{corollary}\label{cor:wisecomplex}
  The group $G$ admits a geometric action on a $k$--systolic complex $W(D(G, C(\abs{\mQ})))$. Thus $G$ is a $k$--systolic group. Moreover, complexes $D(G, C(\abs{\mQ}))$ and $W(D(G, C(\abs{\mQ})))$ are $G$--homotopy equivalent. 
\end{corollary}

\begin{proof}
  All statements follow from combining Theorem~\ref{thm:gcogissystolic} and Theorem~\ref{thm:graphicalissystolic}.
\end{proof}

Being systolic implies many properties of combinatorial flavour, some of which are still unknown even for $2$--dimensional $\czero$ groups. One of such properties is biautomaticity.

\begin{corollary}{\cite[Theorem~E]{JS2}}\label{cor:biautomatic} 
  The group $G$ is biautomatic.
\end{corollary}

A downside of the complex $W(D(G, C(\abs{\mQ})))$ is that it can have arbitrarily large dimension.

\begin{proposition}\label{prop:dimwise}
  We have $\mathrm{dim}(W(D(G, C(\abs{\mQ}))))=\mathrm{max}\{\, \abs{G_v} \mid v \in \mQ \,  \}-1$.
\end{proposition}

\begin{proof}
  The claim follows directly from the definition of the Basic Construction and the Wise complex. 
\end{proof}

 On the other hand, in the following special case one can obtain a $2$--dimensional systolic complex on which $G$ acts geometrically.

\begin{proposition}\label{prop:directsystolic} 
  If every small vertex of $\mQ$ has valence $2$ then $D(G, C(\abs{\mQ}))$ admits a systolic triangulation. 
\end{proposition}

\begin{proof}
  Given any small vertex $v \in D(G, C(\abs{\mQ}))$ let $w_1$ and $w_2$ be the two big vertices of $D(G, C(\abs{\mQ}))$ adjacent to $v$. Remove from $D(G, C(\abs{\mQ}))$ vertex $v$ and then replace edges $[v,w_1]$ and $[v,w_2]$ with a single edge $[w_1,w_2]$. Then given any pair of triangles $[v,w_1,c]$ and $[v,w_2,c]$ for some cone vertex $c \in D(G, C(\abs{\mQ}))$, replace them with a single triangle $[w_1,w_2,c]$. Repeat this procedure for all the small vertices of $D(G, C(\abs{\mQ}))$.

  The above procedure clearly removes all the small vertices of $D(G, C(\abs{\mQ}))$. Moreover, it does not change the links of big vertices. By the proof of Theorem~\ref{thm:developable} these links are $6$--large, as they do not contain cycles of length less than $6$. Now, given any cone vertex $c \in D(G, C(\abs{\mQ}))$, its link before applying the above procedure was isomorphic to $\abs{\mQ}$ (see proof of Theorem~\ref{thm:developable}). One easily sees that the procedure reduces the girth of the link by half. Since by the $k$--hugeness assumption, girth of $\abs{\mQ}$ is equal to $2k$, after the procedure the girth of the link of $c$ is equal to $k \geq 6$.

  Since all vertex links are graphs, $6$--largeness implies that they are flag. Finally, since $D(G, C(\abs{\mQ}))$ supports a $\czero$ metric, it is simply connected, and thus we conclude that $D(G, C(\abs{\mQ}))$ with the simplicial structure defined above is $6$--systolic.
\end{proof}

\begin{remark}
  Proposition~\ref{prop:directsystolic} implies that $D(G, C(\abs{\mQ}))$ admits a $6$--systolic triangulation, regardless of the value of $k$ in the $k$--hugeness assumption on $G(\mQ)$.
\end{remark}

\section{Hyperbolicity}

In this section we investigate when the fundamental group of $G(\mQ)$ is $\delta$--hyperbolic. We begin with the following special case.

\begin{proposition}\label{prop:7systolicishyperbolic}
  Let $G(\mQ)$ be a graphical complex of groups and let $G$ be its fundamental group. If $G(\mQ)$ is $k$--huge for $k \geq 7$ then $G$ is $\delta$--hyperbolic.
\end{proposition}

\begin{proof}
  By Corollary~\ref{cor:wisecomplex} we have that $G$ is $k$--systolic, and $k$--systolic groups for $k \geq 7$ are $\delta$--hyperbolic by \cite[Theorem~A]{JS2}.
\end{proof}

As we will see in Section~\ref{subsec:non-hypeg}, the fundamental group $G$ of a $6$--huge complex $G(\mQ)$ is not necessarily $\delta$--hyperbolic. In this section we show that if $D(G, C(\abs{\mQ}))$ does not contain a certain forbidden configuration of cone-cells, called a \emph{proper triple}, then $G$ is $\delta$--hyperbolic.

\begin{theorem}\label{thm:notripleshyperbolicity}
   Let $G(\mQ)$ be a $6$--huge graphical complex of groups and let $G$ be its fundamental group. If $D(G, C(\abs{\mQ}))$ does not contain a proper triple then $G$ is hyperbolic.
\end{theorem}

We define a proper triple for a general graphical complex $X$.

\begin{definition}[Proper triple] 
  Let $X$ be a simply connected graphical complex and suppose that $X$ satisfies the $C(6)$--condition. Then for every graph $\Gamma_i$, the map $\Gamma_i \to X$ is an embedding \cite[Lemma~6.10.(1)]{OsaPry} and thus we can identify $\Gamma_i$ with its image in $X$. Observe that then the cone-cell corresponding to $\Gamma_i \to X$ is isomorphic to the cone $C(\Gamma_i) \subset X$.

  A triple of cone-cells $C(\Gamma_1), C(\Gamma_2), C(\Gamma_3)$ in $X$ is called \emph{proper} if the triple intersection $C(\Gamma_1) \cap C(\Gamma_2) \cap C(\Gamma_3)$ is non-empty and if it is a proper subset of every double intersection.
\end{definition}

\begin{definition}\label{def:tetrahedron}
  A \emph{cut-up tetrahedron} is a simplicial complex built out of $6$ vertices and the following $4$ triangles: $[v_1,v_2,v_3]$, $[s_1,v_1,v_2]$, $[s_2,v_2,v_3]$, $[s_3,v_1,v_3]$.
\end{definition}

\begin{lemma}\label{lem:notriplenocutup}
  Let $X$ be a simply connected graphical complex satisfying the $C(6)$--condition. If $X$ does not contain a proper triple then $W(X)$ does not contain a cut-up tetrahedron whose $1$--skeleton is isometrically embedded with respect to the edge-path metric.
\end{lemma}

\begin{proof}
  Let $T \subset W(X)$ be a cut-up tetrahedron with simplices as in Definition~\ref{def:tetrahedron}. We will show that $T^{(1)}$ is not isometrically embedded in $W(X)^{(1)}$. For $i \in \{1,2,3\}$ let $V_i,\, S_i$ denote the cone-cells of $X$ corresponding to vertices $v_i,\, s_i$ of $W(X)$. By definition of $T$, the intersection $V_1 \cap V_2 \cap V_3$ is non-empty. Since the triple $V_1, V_2, V_3$ is not proper, at least one of the double intersections, say $V_1 \cap V_2$, is equal to $V_1 \cap V_2 \cap V_3$. Since $S_1 \cap V_1 \cap V_2$ is non-empty and since \[S_1 \cap V_1 \cap V_2 \subseteq V_1 \cap V_2=V_1 \cap V_2 \cap V_3,\] we conclude that $S_1 \cap V_3$ is non-empty. This means that we have an edge $[s_1,v_3]$ in $W(X)$ and thus $T^{(1)}$ is not isometrically embedded in $W(X)^{(1)}$.
\end{proof}

\begin{lemma}\label{lem:nocutupnohyperbolic}
  Let $Y$ be a systolic simplicial complex. If $Y$ does not contain an isometrically embedded $1$--skeleton of a cut-up tetrahedron then $Y$ is $\delta$--hyperbolic with respect to the edge-path metric.
\end{lemma}

\begin{proof}
  Since $Y$ does not contain an isometrically embedded $1$--skeleton of a cut-up tetrahedron, metric triangles in $Y$ (see \cite{CCHO}) have uniformly bounded side length. Thus by \cite[Proposition 9.10]{CCHO} we obtain that $Y$ is $\delta$--hyperbolic.
\end{proof}

\begin{proof}[Proof of Theorem~\ref{thm:notripleshyperbolicity}]
  Let $X=D(G, C(\abs{\mQ}))$. Since $X$ does not contain a proper triple, by Lemma~\ref{lem:notriplenocutup} we get that $W(X)$ does not contain an isometrically embedded $1$--skeleton of a cut-up tetrahedron. Thus by Lemma~\ref{lem:nocutupnohyperbolic} we get that $W(X)$ is $\delta$--hyperbolic. Since by Corollary~\ref{cor:wisecomplex} $G$ acts geometrically on $W(X)$, we conclude that $G$ is $\delta$--hyperbolic.
\end{proof}

\begin{remark}
  Observe that Lemmas~\ref{lem:notriplenocutup} and \ref{lem:nocutupnohyperbolic} do not involve any group acting on $X$ or $Y$. However, if $Y$ admits a geometric action of a group then Lemma~\ref{lem:nocutupnohyperbolic} follows from \cite[Theorem~1.2]{Przyhyp} (see also \cite[Corollary~4.14]{E1}).
\end{remark}

Proper triples in the Basic Construction $D(G, C(\abs{\mQ}))$ have a very simple structure. Recall that, by definition, any cone-cell  in $D(G, C(\abs{\mQ}))$ is of the form $[g, C(\abs{\mQ})]$ for some $g \in G$. 

\begin{lemma}\label{lem:structureofpropertriple}
  A proper triple in $D(G, C(\abs{\mQ}))$ consists of three cone-cells indexed by $g_1$, $g_2$, $g_3$, such that
  \begin{itemize}
    \item the triple intersection is equal to $[g_i, w]$ for some big vertex $w \in \mQ$ and we have $g_i \in G_w$ for $i \in \{1,2,3\}$,
    \item double intersections are equal to respectively $[g_1,\abs{\mQ_{\geq v_1}}]$, $[g_2,\abs{\mQ_{\geq v_2}}]$, and $[g_3,\abs{\mQ_{\geq v_3}}]$, for some distinct small vertices $v_1,v_2,v_3 \in \abs{\mQ_{\leq w}}$ and we have $g_1^{-1}g_3 \in G_{v_1}$, $g_1^{-1}g_2 \in G_{v_2}$, and $g_2^{-1}g_3 \in G_{v_3}$.
  \end{itemize}
\end{lemma}

\begin{proof}
  The possible intersections between cone-cells are equal to $[g, \abs{\mQ_{\geq v}}]$ for some small vertex $v \in \mQ$, or $[g, w]$ for some big vertex $w \in \mQ$. Moreover, the only possibility for having proper inclusions of the triple intersection into double intersections, is that the triple intersection is equal to  $[g, w]$ for some big vertex $w \in \mQ$ and the double intersections are all of the form $[g_i,\abs{\mQ_{\geq v_i}}]$. The remaining claims follow now directly from Definition~\ref{def:basicconstruction}.
\end{proof}

\begin{proposition}\label{prop:2smallneighbourshyp}
  Suppose that $\mQ$ is a poset of simplices of a finite graph (or, equivalently, suppose that for every big vertex $w\in \mQ$ there are exactly two small vertices $v_1, v_2$ with $v_1 \leq w$ and $v_2 \leq w$). Assume that $\mQ$ is $6$--huge, let $G(\mQ)$ be a graphical complex of groups and let $G$ be its fundamental group. Then $G$ is $\delta$--hyperbolic.
\end{proposition} 

\begin{proof} 
  In light of Theorem~\ref{thm:notripleshyperbolicity}, it is enough to show that $D(G, C(\abs{\mQ}))$ does not contain a proper triple. By Lemma~\ref{lem:structureofpropertriple} the existence of a proper triple implies that there is a big vertex $w \in \mQ$ with three distinct small vertices contained in $\abs{\mQ_{\leq w}}$. This contradicts the assumption that there are exactly two small vertices which are smaller than $w$.
\end{proof}

\begin{remark}\label{rem:notriplesisT4}
  The condition of `not containing proper triples' is essentially saying that any closed loop in a link of a big vertex in $D(G, C(\abs{\mQ}))$ passes through at least $4$ cone-cells. Thus, this condition may be seen as a version of a $T(4)$ small cancellation condition. 
\end{remark}

We end this section with the following observation.

\begin{proposition}[Piecewise hyperbolic $\mathrm{CAT}(-1)$ structure on $D(G, C(\abs{\mQ}))$]\label{prop:piecewisehyperbolic}
  Suppose that $G(\mQ)$ is either $k$--huge with $k \geq 7$, or it is $6$--huge and does not contain a proper triple. Then $D(G, C(\abs{\mQ}))$ admits a piecewise hyperbolic $\mathrm{CAT}(-1)$ metric. Thus $G$ is a $\mathrm{CAT}(-1)$ group.
\end{proposition}

\begin{proof}
  In either situation, we alter the metric on $D(G, C(\abs{\mQ}))$ slightly. In the first case, declare all triangles to be isometric to a hyperbolic triangle with angle $\frac{\pi}{2}$ at the small vertex, angle $\frac{\pi}{3}$ at the big vertex, and angle $\frac{\pi}{7}$ at the cone vertex. In the second case, we take a hyperbolic triangle with respective angles $\frac{\pi}{2}$, $\frac{\pi}{4}$,$\frac{\pi}{6}$. The reader easily verifies that with respect to these metrics, the links of $D(G, C(\abs{\mQ}))$ have girth at least $2\pi$, and thus the piecewise hyperbolic metric on $D(G, C(\abs{\mQ}))$ is in fact $\mathrm{CAT}(-1)$.
\end{proof}

Note that being $\mathrm{CAT}(-1)$ is a possibly stronger property than $\delta$--hyperbolicity (see \cite[Proposition~III.H.1.2]{BH}). In particular, the above proposition gives an alternative way of showing that $G$ is $\delta$--hyperbolic.

\section{Examples}\label{sec:examples}

\subsection{Graphical products}

Let $\mQ$ be a $6$--huge $1$--dimensional poset.

\begin{definition}[Graphical product of finite groups]\label{def:graphicalproduct}
  A \emph{graphical product} of finite groups is a graphical complex of groups $G(\mQ)$ over $\mQ$ defined as follows.
  \begin{itemize}
    \item For every small vertex $v$ take an arbitrary finite group $G_v$. 
    \item For every big vertex $w$ define $G_w= G_{v_1} \times G_{v_2} \times \ldots \times G_{v_k}$ where $v_1, v_2,\ldots, v_k$ are small vertices with $v_i \leq w$, and let $\phi_{v_iw}$ be the canonical inclusion of $G_{v_i}$ as a direct factor of $G_w$.
  \end{itemize}
\end{definition}

Definition~\ref{def:graphicalproduct}, while appropriately rephrased, is a special case of a \emph{graph product} of finite groups (see \cite[Example~18.1.10]{Davbook} for a definition).  

\begin{definition}\label{def:graphproduct}
  Let $K$ be a simplicial complex built out of $\abs{\mQ}$ in the following way. For every big vertex $w \in Q$ replace $\abs{\mQ_{\leq w}}$ with a simplex spanned by all small vertices $v \in \abs{\mQ_{\leq w}}$. Note that since $\mQ$ is $6$--huge, this procedure does not result in double edges, and thus $K$ is indeed a simplicial complex (actually, $3$--hugeness would suffice here).

  Observe that vertices of $K$ are precisely the small vertices of $\abs{\mQ}$. Thus we can consider the graph product $G(K)$ over $K$ with groups $G_v$ as vertex groups.
\end{definition}

One easily verifies that fundamental groups of $G(K)$ and $G(\mQ)$ are isomorphic. However, the Basic Construction $D(G, C(K))$ (in this generality defined in \cite{Davbook}) has dimension equal to $\mathrm{max}_{w \in \mQ}  \abs{ \{v  \in \mQ \mid v \leq w \}}$, while $D(G, C(\abs{\mQ}))$ is by definition $2$--dimensional.  

In \cite{PePry} we describe the so-called Bestvina complex, which is a complex resulting from `equivariantly simplifying' complex $D(G, C(K))$. One easily verifies that Bestvina complex associated to $G(K)$ results in a Basic Construction $D(G, B)$ homeomorphic to $D(G, C(\abs{\mQ}))$. 

\begin{example}
  Let $G(\mQ)$ be a graphical product such that $G_v \cong \mathbb{Z}/2$ for every $v \in \mQ$. Then the fundamental group $G$ of $G(\mQ)$ is called the \emph{right-angled Coxeter group} associated to complex $K$ of Definition~\ref{def:graphproduct}.
\end{example}

\begin{proposition}\label{prop:graphicalproducthyp}
  Let $G$ be the fundamental group of a graphical product $G(\mQ)$. Then $G$ is $\delta$--hyperbolic.
\end{proposition}

\begin{proof}
  By Theorem~\ref{thm:notripleshyperbolicity} it is enough to show that $D(G, C(\abs{\mQ}))$ does not contain a proper triple. Assume the contrary. Then by Lemma~\ref{lem:structureofpropertriple} there exists a big vertex $w \in \mQ$ and three elements $g_1, g_2, g_3 \in G_w$ such that for some distinct small vertices $v_1,v_2,v_3 \in \abs{\mQ_{\leq w}}$ we have $g_1^{-1}g_3 \in G_{v_1}$, $g_1^{-1}g_2 \in G_{v_2}$, and $g_2^{-1}g_3 \in G_{v_3}$. Since \[g_2^{-1}g_3= (g_1^{-1}g_2)^{-1} g_1^{-1}g_3\] we have that $g_2^{-1}g_3\in G_{v_1}\times G_{v_2} \leq G_w$.  On the other hand we have $g_2^{-1}g_3 \in G_{v_3}$ which contradicts the fact that $G_w= G_{v_1} \times  G_{v_2} \times  G_{v_3} \times \ldots \times  G_{v_k}$ and that for any $v_i$ the map $G_{v_i} \to  G_{w}$ is the canonical inclusion of a direct factor. 
\end{proof}

\begin{remark}
  Proposition~\ref{prop:graphicalproducthyp} follows immediately from a well-known fact that a graph product of finite groups is $\delta$--hyperbolic if and only if the defining simplicial complex $K$ is $5$--large \cite[Corollary~18.3.10]{Davbook}. One easily sees that if $\mQ$ is $6$--huge then $K$ is $6$--large.

  We decided to include the proof since it is instructive in understanding how to obtain non-hyperbolic examples of graphical complexes of groups.
\end{remark}

\subsection{Coxeter groups with 1--dimensional nerves}

Let $(W,S)$ be a Coxeter system. The \emph{nerve} of $(W,S)$ is a simplicial complex $L$ with vertex set $S$ and simplices being the spherical subsets of $S$. We are interested in the situation where $L$ is $1$--dimensional. Notice that this happens precisely when there is no triple of elements in $S$ that generate a finite  subgroup of $W$. 

Such a Coxeter system $(W,S)$ may be seen as a graphical complex of groups $W(\mQ)$, where $\mQ$ is the poset of simplices of $L$. The local groups at vertices of $L$ are cyclic of order $2$, each generated by an element of $S$. The local groups at edges of $L$ are finite dihedral groups, each generated by the two reflections corresponding to the endpoints of a given edge. The fundamental group of  $W(\mQ)$ is isomorphic to $W$.

Note that $W(\mQ)$ is $k$--huge if and only if $L$ if $k$--large. 
By Proposition~\ref{prop:2smallneighbourshyp}, if $W(\mQ)$ is $k$--huge, then $W$ is hyperbolic.

We remark that the Basic Construction $D(W, C(\abs{\mQ}))$, seen as a simplicial complex, is isomorphic to the Davis complex of $W$. However, the $\czero$ metric on $D(W, C(\abs{\mQ}))$ resulting from Theorem~\ref{thm:developable} is different from the $\czero$ metric on the Davis complex.

\subsection{Non-hyperbolic examples}\label{subsec:non-hypeg}

In order to construct graphical complexes with non-hyperbolic fundamental groups, one has to consider slightly more complicated inclusions of local groups.

\begin{definition}\label{ndef:toroidalgraph}
  Let $\mathbb{E}^2_h$ denote the regular hexagonal tiling of the Euclidean plane, where each hexagon has edge length $1$. Consider two combinatorial (i.e., respecting the tiling) translations $\tr_1, \tr_2$ of $\mathbb{E}^2_h$. Assume that $\tr_1, \tr_2$ are linearly independent. Thus the quotient $\mathbb{E}^2_h/ \langle \tr_1, \tr_2 \rangle$ is homeomorphic to a torus, which we denote by $T_h$. Assume additionally that $\tr_1$ and $\tr_2$ are chosen such that the girth of the $1$--skeleton $T_h^{(1)}$ is at least $6$ (note that this implies that in fact the girth of $T_h^{(1)}$ is equal to $6$).

  Let $\mQ$ be the poset of simplices of $T_h^{(1)}$ ordered by the reverse inclusion. That is, big vertices of $\mQ$ correspond to vertices of $T_h^{(1)}$ and small vertices of $\mQ$ correspond to the midpoints of edges of $T_h^{(1)}$. Clearly we have that $\abs{\mQ}$ is isomorphic to the barycentric subdivision of $T_h^{(1)}$. Observe that $\mQ$ is $6$--huge.
\end{definition}

\begin{definition}[Locally Klein-four complex groups]\label{ndef:kleiniangp}
  Let $\mQ$ be a $1$--dimensional poset such that every big vertex of $\mQ$ has valence $3$. Define a graphical complex of groups $G(\mQ)$ over $\mQ$ as follows:
  \begin{itemize}
    \item for every small vertex $v \in \mQ$ set $G_v$ to be $\mathbb{Z}/2$,
    \item for every big vertex $w \in \mQ$ set $G_w$ to be the Klein four-group $\mathbb{Z}/2 \times \mathbb{Z}/2$,
    \item for every big vertex $w$ and  vertices $v_1, v_2, v_3 \in \abs{\mQ_{\leq w}}$ define the monomorphisms $\phi_{v_iw}$ as the inclusions of three non-trivial elements of $\mathbb{Z}/2 \times \mathbb{Z}/2$.
  \end{itemize}
  
\end{definition}

Recall that a \emph{flat} in a $2$--dimensional $\czero$ space is an isometrically embedded copy of the Euclidean plane $\mathbb{E}^2$.

\begin{proposition}\label{nprop:exwithflats}
   Consider a locally Klein-four complex of groups $G(\mQ)$, where $\mQ$ is a poset described in Definition~\ref{ndef:toroidalgraph}. Let $G$ be the fundamental group of $G(\mQ)$ and let $D(G, C(\abs{\mQ}))$ be the associated Basic Construction. Then $D(G, C(\abs{\mQ}))$ contains a flat. Consequently, $G$ is not $\delta$--hyperbolic.
\end{proposition}

\begin{proof}
  The proof is divided into four steps, some of which are illustrated in Figure~\ref{fig:flat}.
  
  The main idea is to define a locally injective map $p'$ from the barycentric subdivision of the $1$--skeleton of $\mathbb{E}^2_h$ into $\abs{\mQ}$, and then use it to label by elements of $G$ all $12$--cycles (subdivided hexagons) of the $1$--skeleton. Finally, using the labelling one defines an isometric embedding of (barycentrically subdivided) $\mathbb{E}^2_h$ into $D(G, C(\abs{\mQ}))$.\smallskip

  \noindent \textbf{Step \textlabel{1}{step:1}.} The map $p'$.\smallskip

  By Definition~\ref{ndef:toroidalgraph}, $\mQ$ is the poset of simplices of a torus $T_h\cong  \mathbb{E}^2_h/ \langle \tr_1, \tr_2 \rangle$. Let $\Gamma$ denote the barycentric subdivision of the $1$--skeleton of $\mathbb{E}^2_h$. Consider the universal covering $p \colon \mathbb{E}^2_h \to T_h$. By restricting $p$ to $1$--skeleta and passing to the barycentric subdivision, we obtain a map \[p' \colon \Gamma \to \abs{\mQ}.\] The crucial property of $p'$ is that it is locally injective. Observe also that since the girth of $\abs{\mQ}$ is $12$, $p'$ is injective on each $12$--cycle of $\Gamma$. Moreover, note that if two $12$--cycles of $\Gamma$ intersect then the intersection is either a vertex that is mapped by $p'$ to a big vertex of $\abs{\mQ}$, or a pair of consecutive edges that is mapped by $p'$ isomorphically onto $\abs{\mQ_{\geq v}}$ for some small vertex $v \in \abs{\mQ}$.\smallskip 

  \noindent \textbf{Step \textlabel{2}{step:2}.} Labelling the cycles of $\Gamma$ by elements of $G$.\smallskip 
 
  Pick a $12$--cycle $\alpha \subset\Gamma$ and label it by $e \in G$. We will label the remaining $12$--cycles according to the following rule. Given a $12$--cycle $\beta \subset \Gamma$, suppose that there exists a $12$--cycle $\gamma$ which intersects $\beta$ at a pair of edges and which is labelled by element $g \in G$. The intersection $\beta \cap \gamma$ is mapped by $p'$ onto $\abs{\mQ_{\geq v}}$ for some small vertex $v \in \mQ$. Label $\beta$ by $gg_v$ where $g_v$ is the generator of $G_v$. 

  We use the above rule to label all the remaining $12$--cycles of $\Gamma$. The order in which the rule is applied is indicated in Figure~\ref{fig:flat} with a red line.\smallskip 

  \input{flat.tex}

  \noindent \textbf{Step \textlabel{3}{step:3}.} The labelling is consistent.\smallskip

  Suppose that we have two $12$--cycles $\alpha$ and $\beta$ of $\Gamma$ that intersect at a pair of edges, and are labelled by $g_{\alpha}$ and $g_{\beta}$ respectively. Recall that in this situation we have $p'(\alpha  \cap \beta) =\abs{\mQ_{\geq v}}$ for some small vertex $v \in \abs{\mQ}$. We say that $\alpha$ and $\beta$ have \emph{consistent labels} if $g_{\beta}= g_{\alpha}g_v$.

  We claim that all pairs of $12$--cycles of $\Gamma$ have consistent labels. By Step~\hyperref[step:2]{2}, any pair of cycles that was used to define labelling (i.e.,\ any pair whose intersection is crossed by the red line in Figure~\ref{fig:flat}) has consistent labels. For the remaining intersecting pairs we will show the claim recursively, in the order indicated with a blue line in Figure~\ref{fig:flat}. We need the following lemma.

  \begin{lemma*}
    Consider three pairwise intersecting $12$--cycles $\alpha_1$, $\alpha_2$, $\alpha_3$, and assume that pairs $\alpha_1$, $\alpha_2$, and $\alpha_1$, $\alpha_3$ have consistent labels. Then the pair $\alpha_2$, $\alpha_3$ has consistent labels.
  \end{lemma*}

  \begin{proof}
    For $i \in \{1,2,3\}$ let $g_i \in G$ denote the label of $\alpha_i$. The triple intersection $\alpha_1 \cap \alpha_2 \cap \alpha_3$ is a vertex which is mapped by $p'$ to some big vertex $w \in \abs{\mQ}$, and the respective double intersections are mapped by $p'$ isomorphically onto $\abs{\mQ_{\geq v_{12}}}$, $\abs{\mQ_{\geq v_{13}}}$, $\abs{\mQ_{\geq v_{23}}}$ where $v_{12}$, $v_{13}$, $v_{23}$ are the three small vertices of $\abs{\mQ_{\leq w}}$.

    By assumption we have $g_2 =g_1g_{v_{12}}$ and $g_3 =g_1g_{v_{13}}$. We compute \[g_2^{-1}g_3=g_{v_{12}}^{-1}g_1^{-1} \cdot g_1g_{v_{13}}=g_{v_{12}}^{-1}g_{v_{13}}=g_{v_{23}},\] where the last equality follows from the fact that $G_w \cong \mathbb{Z}/2 \times \mathbb{Z}/2$ and $g_{v_{12}}$, $g_{v_{13}}$, and $g_{v_{23}}$ are the three non-trivial elements of $G_w$.
  \end{proof}

  Observe that any vertex $v$ of the blue path determines two $12$--cycles of $\Gamma$ intersecting at a pair of edges. (i.e.,\ the two cycles containing $v$). Enumerate vertices of the blue path with natural numbers and consider the pair of $12$--cycles corresponding to the first vertex (i.e.,\ the pair whose intersection is formed by two edges incident to vertex $\tilde{p}$ in Figure~\ref{fig:flat}). Observe that these two cycles, together with the $12$--cycle that intersects them at pairs of edges incident to vertices $\tilde{o}$ and $\tilde{q}$ respectively, form a triple which satisfies the assumptions of the Lemma. Thus by the Lemma these two cycles have consistent labels.

  Now we proceed by induction `along the blue path'. Consider the pair of $12$--cycles corresponding to the $n$--th vertex of the path, and assume that all the pairs corresponding to vertices $1,2, \ldots, n-1$ have consistent labels. One easily sees that the pair in question is always a part of a triple that satisfies the assumptions of the Lemma.  
  This finishes the proof of the claim.\smallskip 

  \noindent \textbf{Step \textlabel{4}{step:4}.} The embedding of the flat into $D(G, C(\abs{\mQ}))$.\smallskip
 
  Let $F$ denote the complex resulting from attaching to $\Gamma$ a simplicial cone along every of its $12$--cycles. Note that $F$ is isomorphic to the barycentric subdivision of $\mathbb{E}^2_h$, and therefore it carries a flat $\mathrm{CAT}(0)$ metric such that the edges have length $1$ (note that every hexagon of $\mathbb{E}^2_h$ has edge length $2$ with respect to this metric). Define a map \[i \colon F \to D(G, C(\abs{\mQ}))\] by sending a cone $C(\gamma)$ over a $12$--cycle $\gamma$ labelled by element $g$ to a subcone $[g, C(p'(\gamma))]$ of the cone $[g, C(\abs{\mQ})] \subset D(G, C(\abs{\mQ}))$. We claim that this assignment gives a well-defined map. 

  To see this, consider two $12$--cycles $\alpha$ and $\beta$, labelled by $g_1$ and $g_2$ respectively, and suppose that they intersect at a pair of edges which is mapped by $p'$ onto $\abs{\mQ_{\geq v}}$ for some small vertex $v \in \abs{\mQ}$. On one hand, we have $i(\alpha \cap \beta)=[g_1, \abs{\mQ_{\geq v}}]$, and on the other hand we have $i(\alpha\cap \beta)=[g_2, \abs{\mQ_{\geq v}}]$. Since  by Step~\hyperref[step:3]{3} the labelling is consistent, we get that $[g_1, \abs{\mQ_{\geq v}}] =[g_2, \abs{\mQ_{\geq v}}]$. This shows the claim.

  Now one easily checks that $i$ is simplicial (and thus continuous), injective, and locally isometric. Since both $F$ and $D(G, C(\abs{\mQ}))$ are $\czero$, we get that $i$ is an isometric embedding.
\end{proof}

\begin{remark}\label{rem:exwithflatsrelhyp}
  (1) Even if it is not spelled out explicitly, the concept of  a proper triple  plays a crucial role in the proof of Proposition~\ref{nprop:exwithflats}. The consistence of labelling in Step~\hyperref[step:3]{3} is defined in a way that any triple of $12$--cycles having a non-empty intersection is mapped by $i$ to cone-cells that form a proper triple.

  \textlabel{(2)}{rem:exwithflatsrelhyp.2}  It is not very difficult to prove that for $G(\mQ)$ as in Proposition~\ref{nprop:exwithflats}, flats in $D(G, C(\abs{\mQ}))$ are isolated in the sense of \cite{hru}. Intuitively, this is because both poset $\mQ$ and local groups of $G(\mQ)$ are very small (in terms of valence of vertices and number of elements respectively), and so there is not enough room for flats to branch. To show it, one proceeds very similarly as in \cite[Section~7]{wisecubu}.

  Acting geometrically on a $\czero$ $2$--complex with isolated flats implies that $G$ is relatively hyperbolic with respect to the collection of its maximal virtually $\mathbb{Z}^2$ subgroups, and in particular, that $G$ contains $\mathbb{Z}^2$ as a subgroup. It also implies that $G$ is biautomatic and that it satisfies the Tits Alternative (see \cite{hru} for all the above statements). For the sake of completeness, in Definition~\ref{ndef:torusdouble} we give an example of a poset $\mQ$, such that the locally Klein-four complex of groups over $\mQ$ has non-isolated flats in $D(G, C(\abs{\mQ}))$.
\end{remark}

In order to show that $D(G, C(\abs{\mQ}))$ has non-isolated flats, it is enough to find an isometric embedding of a \emph{triplane} into $D(G, C(\abs{\mQ}))$. Recall that a triplane is a union of three Euclidean halfplanes glued together along their boundary lines. To find such an embedding, we employ a similar strategy as in the proof of Proposition~\ref{nprop:exwithflats}. The idea is to construct a poset $\mQ$ whose geometric realisation is, in the sense of Definition~\ref{def:grcomplex}, the $1$--skeleton of a pair of tori identified along a simple closed curve. The universal cover of such space is homeomorphic to the product of a regular $4$--valent tree and a line, and thus one finds a plethora of triplanes inside it.

\begin{definition}\label{ndef:torusdouble}
  Let $T_h$ be a torus arising from Definition~\ref{ndef:toroidalgraph}. Consider a simple closed curve $P \to T_h$ such that in any hexagon $C$ of $T_h$ intersecting $P$, $P$ is a straight line joining the midpoints of a pair of opposite edges of $C$. 

  Let $\mathcal{R}$ be the poset of simplices of $T_h^{(1)}$ ordered by the reverse inclusion, and let $\{v_1,\ldots, v_k\}$ denote the small vertices of $\mathcal{R}$ corresponding to midpoints of edges of $T_h$ crossed by $P$. Let $\mQ= \mathcal{R} \cup_{\{v_1,\ldots, v_k\}}\mathcal{R}$ be the poset obtained from two copies of $\mathcal{R}$ by identifying the respective copies of vertices in $\{v_1,\ldots, v_k\}$.
\end{definition}

\begin{proposition}\label{prop:torusdoubletri}
  Consider a locally Klein-four complex of groups $G(\mQ)$, where $\mQ$ is a poset described in Definition~\ref{ndef:torusdouble}. Let $G$ be the fundamental group of $G(\mQ)$ and let $D(G, C(\abs{\mQ}))$ be the associated Basic Construction. Then $D(G, C(\abs{\mQ}))$ contains an isometrically embedded triplane. Consequently, $G$ is not relatively hyperbolic.
\end{proposition}

\begin{proof}[Sketch of the proof]  
  The approach is analogous to that of Proposition~\ref{nprop:exwithflats}. One verifies that $\abs{\mQ}$ admits a locally injective map from the $1$--skeleton of a `graphical hexagonal triplane'. A graphical hexagonal triplane is a union of three halfplanes glued along their boundary lines, where each halfplane is a subspace of $\mathbb{E}^2_h$ bounded by a straight line passing through midpoints of a pair of opposite edges of some hexagon in $\mathbb{E}^2_h$. 

  Then one can use local injectivity of the aforementioned map to label by elements of $G$ and cone off cycles (and graphs) of the $1$--skeleton, and define an embedding of the resulting complex into  $D(G, C(\abs{\mQ}))$. By \cite[Theorem~5.4]{hru} $G$ is not relatively hyperbolic.
\end{proof}
    
We conclude with a description of flats in Basic Construction $D(G, C(\abs{\mQ}))$. 

\begin{proposition}\label{nprop:structureofflat}
  Assume that $G$ is a $6$--huge graphical complex of groups and let $D(G, C(\abs{\mQ}))$ be the associated Basic Construction. Suppose $F \subset D(G, C(\abs{\mQ}))$ is a flat. Then $F$ is a subcomplex isomorphic to the barycentric subdivision of $\mathbb{E}_h^2$, such that:
  \begin{itemize} 
    \item cone vertices of $F$ correspond to barycenters of hexagons of $\mathbb{E}_h^2$,
    \item small vertices of $F$ correspond to midpoints of edges of $\mathbb{E}_h^2$,
    \item big vertices of $F$ correspond to vertices of $\mathbb{E}_h^2$.
  \end{itemize}
  Moreover, for any three hexagons meeting at a single vertex, their labelling elements of $G$ form a proper triple. 
\end{proposition}

\begin{proof}
  The proof is an elementary exercise and is left to the reader.
\end{proof}

\section{C(4)--T(4) theory}\label{sec:c4t4}

The condition of `not containing proper triples' may be seen as a form of the $T(4)$--condition from classical small cancellation theory (see Remark~\ref{rem:notriplesisT4}). In this section we outline the theory of $C(4)$--$T(4)$ graphical complexes of groups.

Note that in the setting of graphical complexes of groups, by Lemma~\ref{lem:structureofpropertriple} the `no proper triples' condition has a concrete and easily checkable form. Moreover, it can be checked without referring to the Basic Construction, and in particular it does not require $G(\mQ)$ to be developable. This is because the vertex links of the Basic Construction are isomorphic to the local developments, and the latter depend only on $G(\mQ)$. For the sake of clarity we include the definition 

\begin{definition}[$T(4)$--condition]
  Let $G(\mQ)$ be a graphical complex of groups. We say that $G(\mQ)$ satisfies the $T(4)$--condition if for any big vertex $w \in \mQ$ and for any two small vertices $v_1, v_2 \in \mQ_{\leq w}$, and for any two elements $g_1 \in G_{v_1}$ and $g_2 \in G_{v_2}$ there does not exist a vertex $v_3 \leq w$ such that $g_1^{-1}g_2 \in G_{v_3}$.
\end{definition}

Note that the above definition is essentially the same as Lemma~\ref{lem:structureofpropertriple} after `translating its statement by $g_1$'.\medskip

From now on assume that $G(\mQ)$ is a graphical complex of groups that satisfies conditions $C(4)$ (i.e., it is $4$--huge) and $T(4)$, and let $G$ be its fundamental group. Below we present a streamlined account of properties of $G(\mQ)$.

\begin{itemize}
  \item $G(\mQ)$ is non-positively curved. Thus $G(\mQ)$ is strictly developable and $D(G, C(\abs{\mQ}))$ admits a piecewise Euclidean $\czero$ metric. Consequently, $G$ is a $\czero$ group. The proof is analogous to proof of Theorem~\ref{thm:developable}. One metrises cone $ C(\abs{\mQ})$ such that angles at small vertices are $\frac{\pi}{2}$ and angles at big and cone vertices are $\frac{\pi}{4}$.
  
  \item Assuming Convention~\ref{conv:connected}, the Basic Construction $D(G, C(\abs{\mQ}))$ is geodesically complete. Therefore $G$ is infinite. One can construct explicit elements of infinite order in a manner analogous to Proposition~\ref{prop:inforderelements}.
  
  \item $G$ satisfies the Tits Alternative. To prove it in this case, it suffices to define a simplicial map from $D(G, C(\abs{\mQ}))$ to a single triangle with angles  $\frac{\pi}{2}$, $\frac{\pi}{4}$, and $\frac{\pi}{4}$, which restricts to an isometry on every triangle of  $D(G, C(\abs{\mQ}))$. This is done in the same way as in Proposition~\ref{prop:titsalternative}.
\end{itemize}

The main drawback of $C(4)$--$T(4)$ theory is that the Wise complex of Definition~\ref{def:wisecomplex} does not have to be systolic, and thus it cannot be used to show that $G$ is biautomatic. However, in \cite[Theorem~6.18 and Theorem~8.1]{Helly} biautomaticity is shown for all groups acting geometrically on graphical $C(4)$--$T(4)$ complexes.\smallskip

The advantage of $C(4)$--$T(4)$ theory is that one can easily describe examples where $D(G, C(\abs{\mQ}))$ contains flats, as opposed to a rather involved approach in Proposition~\ref{nprop:exwithflats}. For example, the right-angled Coxeter group corresponding to the $4$--gon, seen as a graphical complex of groups, satisfies conditions $C(4)$ and $T(4)$. This group is the product of two infinite dihedral groups, and $D(G, C(\abs{\mQ}))$ is isomorphic to the Euclidean plane with the standard cubical structure. More generally, consider the right-angled Coxeter group corresponding to the complete bipartite graph $K_{n,m}$ for $n,m\geq 3$. This group contains a subgroup isomorphic to $F_2 \times F_2$. Thus the associated Basic Construction $D(G, C(\abs{\mQ}))$ contains non-isolated flats, and therefore $G$ is not relatively hyperbolic.\smallskip

One can ask when $G$ is hyperbolic. We claim that this is the case, for instance, when $G(\mQ)$ satisfies conditions $C(5)$ and $T(4)$. However, in order to show it, one cannot use the proof of Proposition~\ref{prop:7systolicishyperbolic} or Theorem~\ref{thm:notripleshyperbolicity}, as they both rely on the theory of systolic complexes. Instead, one can directly construct a piecewise hyperbolic $\mathrm{CAT}(-1)$ structure on $D(G, C(\abs{\mQ}))$ in the same manner as in Proposition~\ref{prop:piecewisehyperbolic}. One metrises $D(G, C(\abs{\mQ}))$  with a single shape of cells, which is a hyperbolic triangle with angles $\frac{\pi}{2}$, $\frac{\pi}{4}$, and $\frac{\pi}{5}$ at respectively small, big, and cone vertices.

\section{Comments and questions}~\label{sec:commques}

In the article we have explored basic geometric and algebraic properties of graphical complexes of groups. In this section we gather questions and ideas regarding some more advanced properties. Throughout the section let $G(\mQ)$ be a $6$--huge graphical complex of groups, $G$ its fundamental group, and $D(G, C(\abs{\mQ}))$ the associated Basic Construction.

\begin{enumerate}
  \item Is $G$ virtually torsion free? Observe that it is the case for graphical products, and in particular for right-angled Coxeter groups. Proposition~\ref{prop:inforderelements} provides plenty of infinite order elements, yet, in order to obtain a torsion free subgroup $G' \leq G$ with $D(G, C(\abs{\mQ}))/G'$ compact, one needs to include additional, more complicated infinite order elements. Intuitively, the axes of these elements should pass through big vertices of $D(G, C(\abs{\mQ}))$. 

  \item Is $G$ residually finite? We do not have any intuition regarding this, but it seems like a reasonable question to ask.

  \item Can $G$ contain a subgroup isomorphic to $F_2 \times F_2$? Notice that in general a systolic group $G$ can contain $F_2 \times F_2$, however, the minimal dimension of a systolic complex acted upon such $G$ is equal to $3$. On the other hand $F_2 \times F_2$ admits an action on a $2$--dimensional $C(6)$--graphical small cancellation complex.

  \item Can $G$ have Kazhdan's property $\mathrm{(T)}$? We point out that for non-hyperbolic groups described in Subsection~\ref{subsec:non-hypeg}, the Basic Construction $D(G, C(\abs{\mQ}))$, together with its hexagonal tiling of flats, resembles an $\tilde{A_2}$--building. Some lattices in such buildings are known to have property $\mathrm{(T)}$.

  \item Under which conditions can $G$ be cubulated? We are not aware of an example of $G$ that cannot act properly (and cocompactly) on a $\czero$ cube complex. Recall that if $G$ is cubulated then it cannot have property $\mathrm{(T)}$. One idea for cubulating $G$ is to start with a poset $\mQ$ such that small vertices of $\mQ$ have valence $2$ and $\abs{\mQ}$ has a structure of a wall space. The walls in $\abs{\mQ}$ naturally determine walls in $D(G, C(\abs{\mQ}))$. These walls are then preserved by the $G$--action and thus we get a cubulation. To obtain examples where $\abs{\mQ}$ has a structure of a wall space, one begins with an arbitrary $\mQ$ and then takes its $\mathbb{Z}/k$--homology cover for $k \geq 2$, together with the assignment of local groups such that the covering map becomes a morphism of graphical complexes of groups. Note that the girth of a finite cover of a graph is always larger than the  girth of the graph, and hence in majority of cases this method leads to hyperbolic examples.

  \item Does $D(G, C(\abs{\mQ}))$ satisfy the Flat Closing Conjecture? Recall that this conjecture states that if $D(G, C(\abs{\mQ}))$ contains a flat then $G$ contains a subgroup isomorphic to $\mathbb{Z}^2$. We remark that the group appearing in Proposition~\ref{nprop:exwithflats} does contain a copy of $\mathbb{Z}^2$ (see Remark~\ref{rem:exwithflatsrelhyp}(\hyperref[rem:exwithflatsrelhyp.2]{2})). One can show that so does the group appearing in Proposition~\ref{prop:torusdoubletri}. 
  
  Moreover, we expect the answer to be positive in the full generality. This is because the structure of flats in $D(G, C(\abs{\mQ}))$ seems to be very restricted and closely linked to the combinatorial structure of $G(\mQ)$. Since the latter is given in terms of finitely many local groups, we think that it is very unlikely for $D(G, C(\abs{\mQ}))$ to contain only aperiodic flats.
\end{enumerate}

\end{document}

%% file: poset.tex
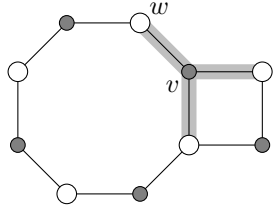
\begin{figure}[!h]
\centering
\begin{tikzpicture}[scale=0.65]

\begin{scope}

\draw [gray!50, line width=2mm] (4,2.5)--(2.5,2.5);
\draw [gray!50, line width=2mm] (2.5,1)--(2.5,2.5)--(1.5,3.5);

\draw (0,0)--(1.5,0)--(2.5,1)--(2.5,2.5)--(1.5,3.5)--(0,3.5)--(-1,2.5)--(-1,1)--(0,0);


\draw (2.5,1)--(4,1)--(4,2.5)--(2.5,2.5);

\draw[fill=white] (0,0) circle [radius=0.2];
\draw[fill=white] (2.5,1) circle [radius=0.2];
\draw[fill=white] (1.5,3.5) circle [radius=0.2];
\draw[fill=white] (-1,2.5) circle [radius=0.2];
\draw[fill=white] (4,2.5) circle [radius=0.2];

\draw[fill=gray] (1.5,0) circle [radius=0.15];
\draw[fill=gray] (2.5,2.5) circle [radius=0.15];
\draw[fill=gray] (0,3.5 ) circle [radius=0.15];
\draw[fill=gray] (-1,1) circle [radius=0.15];
\draw[fill=gray] (4,1) circle [radius=0.15];



\node [below left] at (2.5,2.5) {$v$}; 
\node [above right] at (1.5,3.5) {$w$};

\end{scope}

\end{tikzpicture}
\caption{Geometric realisation of a $1$--dimensional poset $\mQ$.  Big vertices are white, small vertices are dark gray. Subgraph $\abs{\mQ_{\geq v}}$ is light gray. We have $w \in \abs{\mQ_{\geq v}}$.}  
\label{fig:1dimposet}
\end{figure} 

%% file: metric.tex
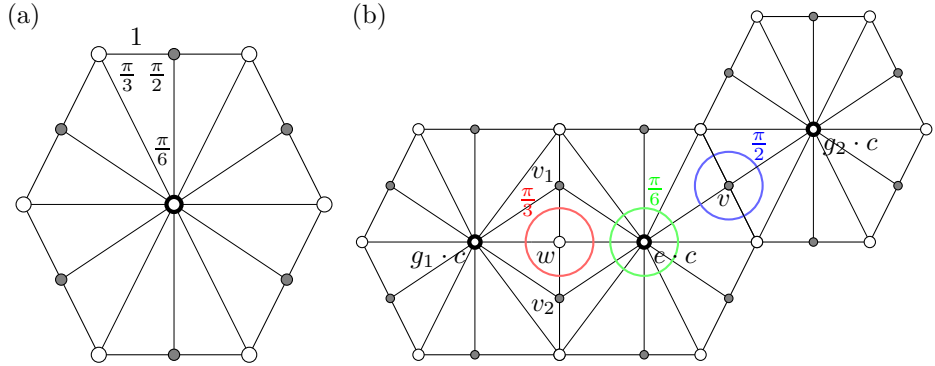
\begin{figure}[!h]
\centering
\begin{tikzpicture}[scale=1]

\node    at (-1,4.5)   {(a)};
\node    at (3.6,4.5)   {(b)};

\begin{scope}

\draw (0,0)--(2,0)--(3,2)--(2,4)--(0,4)--(-1,2)--(0,0);

\draw[ultra thin] (0,0)--(1,2);

\draw[ultra thin] (1,2)--(2,0);

\draw[ultra thin] (1,2)--(3,2);

\draw[ultra thin] (1,2)--(2,4);

\draw[ultra thin] (1,2)--(0,4);

\draw[ultra thin] (1,2)--(-1,2);

\draw[ultra thin] (1,2)--(1,0);

\draw[ultra thin] (1,2)--(2.5,1);

\draw[ultra thin] (1,2)--(2.5,3);

\draw[ultra thin] (1,2)--(1,4);

\draw[ultra thin] (1,2)--(-0.5,3);

\draw[ultra thin] (1,2)--(-0.5,1);


\draw[ultra thick,fill=white] (1,2) circle [radius=0.1];

\draw[fill=white] (0,0) circle [radius=0.1];
\draw[fill=white] (2,0) circle [radius=0.1];
\draw[fill=white] (3,2) circle [radius=0.1];
\draw[fill=white] (2,4) circle [radius=0.1];
\draw[fill=white] (0,4) circle [radius=0.1];
\draw[fill=white] (-1,2) circle [radius=0.1];

\draw[fill=gray] (1,0) circle [radius=0.075];
\draw[fill=gray] (2.5,1) circle [radius=0.075];
\draw[fill=gray] (2.5,3) circle [radius=0.075];
\draw[fill=gray] (1,4) circle [radius=0.075];
\draw[fill=gray] (-0.5,3) circle [radius=0.075];
\draw[fill=gray] (-0.5,1) circle [radius=0.075];

\node [below left]   at (1.025,4)   {$\frac{\pi}{2}$};

\node [above left]   at (1.1,2.4)   {$\frac{\pi}{6}$};
\node [below right]   at (0.125,4)   {$\frac{\pi}{3}$};
\node [above]   at (0.5,4)   {$1$};



\end{scope}

\begin{scope}[shift={(6.5,0)}, scale=0.75]

\begin{scope}[shift={(-3,0)}]

\draw (0,0)--(2.5,0)--(2.5,4)--(0,4)--(-1,2)--(0,0);

\draw[ultra thin] (0,0)--(1,2);

\draw[ultra thin] (1,2)--(2.5,0);

\draw[ultra thin] (1,2)--(2,2);

\draw[ultra thin] (1,2)--(2.5,4);

\draw[ultra thin] (1,2)--(0,4);

\draw[ultra thin] (1,2)--(-1,2);

\draw[ultra thin] (1,2)--(1,0);

\draw[ultra thin] (1,2)--(2.5,1);

\draw[ultra thin] (1,2)--(2.5,3);

\draw[ultra thin] (1,2)--(1,4);

\draw[ultra thin] (1,2)--(-0.5,3);

\draw[ultra thin] (1,2)--(-0.5,1);


\draw[ultra thick,fill=white] (1,2) circle [radius=0.1];

\draw[fill=white] (0,0) circle [radius=0.1];
\draw[fill=white] (2.5,0) circle [radius=0.1];
\draw[fill=white] (3,2) circle [radius=0.1];
\draw[fill=white] (2.5,4) circle [radius=0.1];
\draw[fill=white] (0,4) circle [radius=0.1];
\draw[fill=white] (-1,2) circle [radius=0.1];

\draw[fill=gray] (1,0) circle [radius=0.075];
\draw[fill=gray] (2.5,1) circle [radius=0.075];
\draw[fill=gray] (2.5,3) circle [radius=0.075];
\draw[fill=gray] (1,4) circle [radius=0.075];
\draw[fill=gray] (-0.5,3) circle [radius=0.075];
\draw[fill=gray] (-0.5,1) circle [radius=0.075];


\node [below left]   at (1,2)   {$g_1 \cdot c$};

\end{scope}

\begin{scope}

\draw[fill=white] (0,0)--(2,0)--(3,2)--(2,4)--(-0.5,4)--(-0.5,3)--(-0.5,1)--(-0.5,0)--(0,0);



\draw[ultra thin] (-0.5,0)--(1,2);

\draw[ultra thin] (1,2)--(2,0);

\draw[ultra thin] (1,2)--(3,2);

\draw[ultra thin] (1,2)--(2,4);

\draw[ultra thin] (1,2)--(-0.5,4);

\draw[ultra thin] (1,2)--(-1,2);

\draw[ultra thin] (1,2)--(1,0);

\draw[ultra thin] (1,2)--(2.5,1);

\draw[ultra thin] (1,2)--(2.5,3);

\draw[ultra thin] (1,2)--(1,4);

\draw[ultra thin] (1,2)--(-0.5,3);

\draw[ultra thin] (1,2)--(-0.5,1);


\draw[ultra thick,fill=white] (1,2) circle [radius=0.1];

\draw[fill=white] (-0.5,0) circle [radius=0.1];
\draw[fill=white] (2,0) circle [radius=0.1];
\draw[fill=white] (3,2) circle [radius=0.1];
\draw[fill=white] (2,4) circle [radius=0.1];
\draw[fill=white] (-0.5,4) circle [radius=0.1];
\draw[fill=white] (-0.5,2) circle [radius=0.1];

\draw[fill=gray] (1,0) circle [radius=0.075];
\draw[fill=gray] (2.5,1) circle [radius=0.075];
\draw[fill=gray] (2.5,3) circle [radius=0.075];
\draw[fill=gray] (1,4) circle [radius=0.075];
\draw[fill=gray] (-0.5,3) circle [radius=0.075];
\draw[fill=gray] (-0.5,1) circle [radius=0.075];



\draw[thick,green!60] (1,2) circle [radius=0.6];

\node [below right]   at (1,2)   {$e \cdot c$};

\draw[thick,blue!60] (2.5,3) circle [radius=0.6];

\node [below left]   at (2.7,3)   {$v
$};

\draw[thick,red!60] (-0.5,2) circle [radius=0.6];

\node [below left]   at (-0.4,2)   {$w
$};

\node [above left]   at (-0.4,3-0.1)   {$v_1
$};
\node [below left]   at (-0.4,1.15)   {$v_2
$};

\node [above right]   at (-1.4,2.3)   {$\textcolor{red}{\frac{\pi}{3}
}$};

\node [above right]   at (0.85,2.5)   {$\textcolor{green}{\frac{\pi}{6}
}$};

\node [above right]   at (2.7,3.3)   {$\textcolor{blue}{\frac{\pi}{2}
}$};

\end{scope}

\begin{scope}[shift={(3,2)}]

\draw (0,0)--(2,0)--(3,2)--(2,4)--(0,4)--(-1,2)--(0,0);

\draw[ultra thin] (0,0)--(1,2);

\draw[ultra thin] (1,2)--(2,0);

\draw[ultra thin] (1,2)--(3,2);

\draw[ultra thin] (1,2)--(2,4);

\draw[ultra thin] (1,2)--(0,4);

\draw[ultra thin] (1,2)--(-1,2);

\draw[ultra thin] (1,2)--(1,0);

\draw[ultra thin] (1,2)--(2.5,1);

\draw[ultra thin] (1,2)--(2.5,3);

\draw[ultra thin] (1,2)--(1,4);

\draw[ultra thin] (1,2)--(-0.5,3);

\draw[ultra thin] (1,2)--(-0.5,1);


\draw[ultra thick,fill=white] (1,2) circle [radius=0.1];

\draw[fill=white] (0,0) circle [radius=0.1];
\draw[fill=white] (2,0) circle [radius=0.1];
\draw[fill=white] (3,2) circle [radius=0.1];
\draw[fill=white] (2,4) circle [radius=0.1];
\draw[fill=white] (0,4) circle [radius=0.1];
\draw[fill=white] (-1,2) circle [radius=0.1];

\draw[fill=gray] (1,0) circle [radius=0.075];
\draw[fill=gray] (2.5,1) circle [radius=0.075];
\draw[fill=gray] (2.5,3) circle [radius=0.075];
\draw[fill=gray] (1,4) circle [radius=0.075];
\draw[fill=gray] (-0.5,3) circle [radius=0.075];
\draw[fill=gray] (-0.5,1) circle [radius=0.075];



\node [below right]   at (1,2)   {$g_2 \cdot c$};

\end{scope}

\end{scope}

\end{tikzpicture}
\caption{(a) Piecewise linear metric on $C(\abs{\mQ})$, where $\mQ$ is a poset of simplices of a $6$--cycle. (b) Local developments at vertices $c$, $v$ and $w$. At vertices $c$ and $v$ we  have legal situations, where both green and blue cycles have angular length $2\pi$. At vertex $w$ we have an illegal situation: the red cycle has angular length $ \frac{4\pi}{3}$. In this situation we have $g_1 \in \psi_{v_1w}(G_{v_1})\cap  \psi_{v_2w}(G_{v_2})$.}  
\label{fig:metric}
\end{figure}